\documentclass[12pt]{article}

\usepackage{amsfonts, amsthm, amsmath, amssymb, graphicx, verbatim}

\pdfoutput=1

\newcommand{\R}{\mathbb{R}}
\newcommand{\C}{\mathbb{C}}

\newcommand{\cP}{\mathcal{P}}
\newcommand{\cM}{\mathcal{M}}
\newcommand{\x}{\mathbf{x}}
\newcommand{\y}{\mathbf{y}}
\newcommand{\w}{\mathbf{w}}

\newcommand{\lims}{\mathop{\overline{\lim}}}
\renewcommand{\Im}{\mathop{\text{Im}}}
\newcommand{\placehold}{\mbox{\boldmath$\cdot$}}

\theoremstyle{plain}
\newtheorem{theorem}{Theorem}

\newtheorem{corollary}[theorem]{Corollary}
\newtheorem{lemma}[theorem]{Lemma}

\newtheorem*{question}{Question}
\newtheorem*{thm1}{Theorem \ref{generalmarr}}
\newtheorem*{thm4}{Theorem \ref{thm:momentcontinuity}}
\newtheorem*{generic}{Genericity Condition}
\newtheorem*{marr}{Marr Conjecture}

\newtheorem*{cor2a}{Corollary \ref{cor:OneGaussian}(a)}
\newtheorem*{edgethm}{Theorem \ref{subset}}
\newtheorem*{finitethm}{Theorem \ref{finitemomentexp}}

\theoremstyle{definition}

\begin{document}

\title{The Marr Conjecture and Uniqueness of Wavelet Transforms}

\author{Benjamin Allen\footnote{Dept. of Mathematics, Emmanuel College, Boston, MA 02215} \footnote{Program for Evolutionary Dynamics, Harvard University, Cambridge, MA, 02138} \footnote{Center for Mathematical Sciences and Applications, Harvard University, Cambridge, MA 02138} \ \  and Mark Kon\footnote{Dept. of Mathematics, MIT, Cambridge, MA 02139} \footnote{Dept. of Mathematics and Statistics, Boston University, Boston, MA 02215}}


\maketitle

\begin{abstract}
The inverse question of identifying a function from the nodes (zeroes) of its wavelet transform arises in a number of fields.  These include whether the nodes of a heat or hypoelliptic equation solution determine its initial conditions, and in mathematical vision theory the Marr conjecture, on whether an image is mathematically determined by its edge information.  We prove a general version of this conjecture by reducing it to the moment problem, using a basis dual to the Taylor monomial basis $x^\alpha$ on $\R^n$.
\end{abstract}

\section{Introduction}

\subsection{Background}

The inverse problem of determining a function $f$ from the nodes (zeroes) of its wavelet transform has various applications.  In partial differential equations this becomes the question of recovering the solution of a heat or hypoelliptic equation from its nodes.  In mathematical vision theory it is a generalization of the problem known as the Marr conjecture, about the unique determination of a function from its multiscale edges.  Here we give sufficient conditions on the wavelet and the function $f$ for its recovery, and show that these conditions are the best of their kind.

There has been both theoretical \cite{Meyer,Meyer2,Yuille,CSO,Hummel,Saito,Babaud}  and empirical \cite{Mallat} evidence related to the Marr conjecture, regarding both its range of validity and some restrictions on its scope. As shown by Meyer originally \cite{Meyer}, the truth of the Marr conjecture has limitations, and it is in general false for non-decaying $f$.  

It is shown here that for compactly supported or exponentially decaying $f$, the conjecture holds in a general form; however, it is false for algebraically decaying $f$.

The Marr conjecture was originally motivated by the fact that visual images are in practice often easy to reconstruct from their edges.  To this extent these results are a mathematical formalization of this fact.  In one dimension we apply our results to the Richter (Mexican hat) wavelet, which was the original convolving function studied by Marr \cite{Marr,Marrbook}.

The methods involve reducing the recovery of $f$ to the moment problem, using the duality of two bases for functions on $\R^n$, the Taylor monomials $x^\alpha$ and the derivatives $\delta^{(\alpha)}$ of the delta distribution at $0$. The method of moments provides a natural approach to this problem, as the effects of different moments become asymptotically separated under the wavelet transform. 

The standard $d$-dimensional continuous wavelet transform of $f$ with a smooth wavelet $\tilde\psi$ has the form 
\[
Wf(\sigma,\x)= \sigma^{d/2} \int_{\R^d}f(\mathbf{t}) \; \tilde\psi\left(\frac{\mathbf{t}-\x}{\sigma}\right)d\mathbf{t}=
\sigma^{d/2}f * \psi_\sigma(\x),
\]
where we define 
$\psi(\x)=\tilde\psi(-\x)$, and
$\psi_\sigma(\x)=\sigma^{-d}\psi(\x/\sigma)$ (this notation is used for later convenience).

We ask under what conditions a locally integrable function $f$ is uniquely determined (up to a constant multiple) by the nodes of its wavelet transform. It is in fact possible to answer a stronger version of this question, namely whether $f$ can be recovered from knowledge of the nodes of $Wf(\sigma,\x)$ at 
an arbitrary discrete sequence of scales $\{ \sigma_i \}_{i \geq 0}$.  

This type of question arises in a number of fields:
\begin{itemize}  
\item In wavelet theory, this is an inverse problem for the continuous wavelet transform \cite{Mallat,MallatBook,Meyer,Meyer2}, and the dyadic transform \cite{Mallat,MallatBook} (which is continuous in the space variable $\x$ but discrete in the scaling variable $\sigma$).
\item In mathematical vision theory \cite{Marrbook}, the function $f$ represents an image.  Convolutions of $f$ with rescalings of $\psi(\x) = G(\x) = (2\pi)^{-d/2}e^{-|\x|^2/2}$ represent Gaussian kernel smoothings (blurrings) of the image at different scales, which eliminate small features and maintain large ones.  Defining the Ricker (Mexican hat) wavelet $M(\x)$ as the Laplacian of $G(\x)$, it follows that the zeros of $f*M_\sigma(\x)$ represent points of maximal change in the smoothed image, which can be interpreted as edges of $f$ at scale $\sigma$ (generalized discontinuities).  Thus the nodes of $f*M_\sigma(\x)$ as $\sigma$ increases can be interpreted as successively sparser ``line sketches" of the image $f$.  The unique determination question (Marr conjecture) asks whether these nodes (edges) form a complete representation of the image.  The traditional focus on this question in mathematical vision theory has been based on the widespread use of edge perception as a model for vision.
\item For hypoelliptic partial differential equations, scaled smoothing functions often arise as fundamental solutions (Green's functions).  For example, the Gaussian function $u(x,t) = (2\pi t)^{-d/2}e^{-|\x|^2/2t}$ is the fundamental solution of the heat equation $u_t = \frac{1}{2} \Delta u$.  The solution to an initial value problem is obtained by convolution of the initial condition with the fundamental solution.  The question is then whether the nodes of a solution uniquely determine it.
\end{itemize}

In wavelet theory this question has been studied theoretically and
numerically by Mallat \cite{Mallat,MallatBook} and Meyer \cite{Meyer,Meyer2}, and the mathematical question in vision theory has also received a good deal of attention \cite{Marr, Marrbook,Yuille,CSO,Hummel,Saito,Babaud}.  Although the problem of determining nodes of parabolic equations and their properties has been studied in a number of settings \cite{angenent1988zero,lin1991nodal,watanabe2005zero}, the inverse problem of determining a solution from its nodes has received less attention. 

The mathematical conjecture in vision theory, known as the Marr conjecture \cite{Marr, Marrbook}, is motivated by problems of edge detection and image reconstruction in biological and artificial neural systems.  In this setting it is natural to restrict to functions $f$ that are compactly supported, or more generally, satisfy some decay condition.  The conjecture can be stated as

\begin{marr}
A locally integrable function $f$ of sufficiently rapid decay is uniquely
 determined, up to a constant multiple, by the zero sets of $f * M_{\sigma_i}$ for any sequence of positive scales $\{ \sigma_i \}_{i=1}^\infty$ tending to infinity.
\end{marr}

This conjecture has remained open, although special cases have been proved \cite{Yuille, CSO}.  The corresponding statement for nondecaying functions was disproved by Meyer \cite{Meyer}, who found distinct periodic functions whose Ricker wavelet transforms have identical zero sets at all scales.  

More generally, we can ask for minimal conditions on a general wavelet $\psi$ allowing for such unique determination:

\begin{question}
What conditions on a twice-differentiable function $\psi$ are necessary and sufficient to imply that any function $f$, of sufficiently rapid decay, is uniquely determined up to a constant multiple by (a) the zeros in $(\sigma,\x)$ of $Wf(\sigma,\x)$ (b) the zero sets of $f*\psi_{\sigma_i}$, for any sequence of positive scales $\{ \sigma_i \}_{i=1}^\infty$ tending to infinity.
\end{question}

\subsection{Results on unique determination}

Here we answer this question by finding conditions on $f$ and $\psi$ that are sufficient and the best of their type for such unique determination.  We require that $f$ be integrable and of exponential order---meaning that $f$ belongs to a class $\cP_\gamma'$ of exponentially decaying functions.  We require that $\psi$ belong to a class $\cP$ of smooth functions whose derivatives grow slower than exponentially, and satisfy the following:

\begin{generic}
The regular zero set of any derivative of fixed order $n$ is not contained in the zero set of any other derivative of fixed order $m$, for any $n,m \geq 0$.
\end{generic}

A regular (transverse) zero of a function $\psi$ is a point in all of whose neighborhoods $\psi(\x)$ takes both positive and negative values.
By ``derivative of fixed order $m$", we mean a linear combination of partial derivatives of $\psi$ of order $m$, (i.e., a homogeneous linear differential operator of order $m$ applied to $\psi$), modulo multiplication by a nonzero constant.  As an example, the one-dimensional Gaussian wavelet $G(x)$ fails this genericity condition, in that the regular zero set of $G$ is empty and is therefore trivially contained in the zero set of $G^{(n)}$ for any $n > 0$.  However its second derivative, the Ricker wavelet $M(x)$, satisfies this condition, as we show in Section \ref{gaussianproof}.

Our main result can be stated as follows:
\begin{theorem}
\label{generalmarr}
Given $\psi \in \cP$ satisfying the above genericity condition, any $f \in \cP_\gamma' \cap L^1(\R^d)$ is uniquely determined, up to a constant multiple, by the zero sets of its wavelet transform $f *\psi_{\sigma_i}$ for any sequence of positive scales $\{ \sigma_i \}_{i=1}^\infty$ tending to infinity.
\end{theorem}

We will show that the conditions in this theorem are the best of their kind, in the following sense.  First, the theorem fails if the exponential decay condition $f\in \cP_\gamma$ is weakened to algebraic decay (see Section \ref{necessity}), although this leaves the conjecture open for the restricted set of functions $f$ with decay that is between algebraic and exponential, e.g. $f(x)=e^{-|x|^{1/2}}$  Second, if the genericity condition on the regular zeroes of the wavelet $\psi$ (see above) fails weakly, then the theorem fails to hold (see Section \ref{generalproof}).

\begin{corollary}  Given $\psi$ and $f$ as above, $f$ is uniquely determined by the zero sets of its continuous wavelet transform $f*\psi_{\sigma_i}$ for $\sigma >0$, and more generally its dyadic wavelet transform $f*\psi_\sigma$ restricted to $\sigma = 2^i, i \in \mathbb{N}$. 
\end{corollary}  

In the case of the Ricker (Gaussian derivative) wavelet, we prove the following:

\begin{corollary}
\label{cor:OneGaussian}
\emph{(Marr conjecture in one dimension)} 
\renewcommand{\labelenumi}{(\alph{enumi})}
\begin{enumerate}
\item Any $f \in \cP_\gamma' \cap L^1(\R^d)$ is uniquely determined, up to a constant multiple, by the zero sets of $ f *M_{\sigma_i}$ for any sequence of positive scales $\{ \sigma_i \}_{i=1}^\infty$ with a nonzero limit point.
\item This unique determination fails if the only limit point of $\{ \sigma_i \}_{i=1}^\infty$ is zero.
\item This unique determination also fails if $f$ is of algebraic rather than exponential order.
\end{enumerate}
\end{corollary}

For dimensions $d>1$, the above theorem reduces the Marr conjecture to a statement about polynomial zeros.  For any multiindex of nonegative integers $\alpha = (\alpha_1, \ldots, \alpha_d)$, we define the \emph{Laplace-Hermite polynomial} $L_\alpha(\x)$ in $\x = (x_1, \ldots, x_d)$ by 
\begin{equation}
\label{eq:LHrelation}
\Delta G^{(\alpha)}(\x) = (-1)^{|\alpha|} L_\alpha(\x) G(\x).
\end{equation}
Above, the superscript $(\alpha)$ indicates a mixed partial derivative in the orders specified by $\alpha$. 
Note that $L_\alpha$ is a polynomial of degree $|\alpha|+2$, where $|\alpha|=\alpha_1 + \ldots+\alpha_d$.  We thus have:

\begin{corollary}
\label{cor:dGaussian}
\emph{(Marr conjecture in $d$ dimensions)} 
If there is no pair of distinct Laplace-Hermite polynomials of degree greater than zero such that the zero set of one contains the zero set of the other, then any $f \in \cP_\gamma' \cap L^1(\R)$ is uniquely determined, up to a constant multiple, by the zero sets of $f *M_{\sigma_i}$ for any sequence of positive scales $\{ \sigma_i \}_{i=1}^\infty$ tending to infinity.
\end{corollary}

Thus in any dimension the Marr conjecture is equivalent to a condition on the zeros of Laplace-Hermite polynomials.

\subsection{Results on asymptotic moment expansions}
\label{sec:intromoments}

Our approach is based on moment expansions, which rely on the duality of the basis of Taylor monomials $\x^\alpha = x_1^{\alpha_1}\ldots x_d^{\alpha_d}$ in $\R^d$, with distributions $\delta^{(\alpha)}$ localized at the origin.  Here $\delta^{(\alpha)}$ denotes a distributional partial derivative of the Dirac distribution $\delta$ in the orders specified by the multiindex $\alpha$.  The moment expansion represents a function as a series in $\delta^{(\alpha)}$, with coefficients in terms of the function's moments.  

Moment expansions have been used to study electromagnetism (in multipole expansions), gravitation, and acoustics.  They have more recently also been applied to the Navier-Stokes \cite{Wayne1, Uminsky} and other differential equations \cite{Loura2, Kanwal1, Kanwal2, Hernandez, Wiener}.   Recently, a formalism for asymptotic moment expansions has been developed \cite{Estrada}, in which the moment expansion converges as an asymptotic series.

We extend the theory of asymptotic moment expansions in two ways.  First, we prove the following continuity result for convolutions of moment expansions:
\begin{theorem}
\label{thm:momentcontinuity}
If $f$ is replaced by its asymptotic moment expansion in the convolution $f * \psi_\sigma (\sigma \w)$, the asymptotic convergence of the resulting series, as $\sigma \to \infty$, is locally uniform in $\w$.
\end{theorem}
Second, we generalize the theory of asymptotic moment expansions to distributions with only finitely many moments:
\begin{theorem} 
\label{finitemomentexp}
If the first $n$ moments of $f$ are well-defined, then $f$ has an asymptotic moment expansion to order $n-1$.  If $f$ is replaced by this moment expansion in the convolution $f * \psi_\sigma (\sigma w)$, the asymptotic convergence of the resulting series, as $\sigma \to \infty$, is locally uniform in $w$.
\end{theorem}

\subsection{Results on the geometry of heat equation nodes}

Our work leads to new results on the nodes of solutions to the heat equation initial value problem:
\begin{equation}
\label{heateqn}
\begin{cases} F_t = \frac{1}{2}  F_{xx} & x \in (-\infty, \infty),\, t\in [0, \infty)\\
F(x,0) = f(x). \end{cases}
\end{equation}
The nodes (zeros) of $F_{xx}$ form algebraic curves which we call \emph{edge contours} of $f$. We show that new edge contours do not appear as $t$ increases, strengthening and complementing previous results \cite{angenent1988zero,Yuille,Babaud, Hummel,lin1991nodal,watanabe2005zero}: 

\begin{theorem}
\label{subset}
For an integrable function $f$ of exponential order and for positive numbers $t_1 < t_2$, the edge contours of $f$ intersecting the line $t=t_2$ are a subset of those that intersect $t=t_1$.
\end{theorem}

We also obtain the following unique determination result:

\begin{theorem}
\label{thm:heat}
Let $F$ be a solution to \eqref{heateqn} for some initial condition $f \in L^1(\R)$.  If it is known that the second integral
\begin{equation*}
a(x) = \int_{-\infty}^x \int_{-\infty}^y f(z) \; dz \, dy
\end{equation*}
is a function of exponential order, then $f$ is uniquely determined by the zeros of $F(x,t_j)$ for any sequence
$\{t_j\}_{j=1}^\infty$ of positive real numbers with a positive or infinite limit point.
\end{theorem}


\section{Moment Expansion}
\label{moment}

Moment expansions represent functions (generally distributions) as 
series in derivatives 
\[
\delta^{(\alpha)}(\x) \equiv \frac{\partial^{|\alpha|}}{\partial x_1^{\alpha_1} \ldots \partial x_d^{\alpha_d}} \delta(\x),
\]
based on the fact that these derivatives
and the monomials 
\[
\x^\alpha \equiv x_1^{\alpha_1} \ldots x_d^{\alpha_d}
\]
form a biorthogonal system:
\begin{equation}
\label{biorthogonal}
\langle \delta^{(\alpha)}, \x^\beta \rangle = \begin{cases} (-1)^{|\alpha|} \alpha! & \alpha = \beta \\ 0 & \text{otherwise}, \end{cases}
\end{equation} 
with
\[
|\alpha|=\alpha_1 + \ldots + \alpha_d, \qquad \alpha! = \alpha_1! \ldots \alpha_d!.
\]
In principle, the moment expansion of a distribution $f$ is the series
\begin{equation}
\label{momentprinciple}
f(\x) = \sum_{|\alpha| \geq 0}  \frac{(-1)^{|\alpha|}}{\alpha !} \mu_\alpha \;\delta^{(\alpha)}(\x),
\end{equation}
where $\mu_\alpha$ is the $\alpha$th moment of $f$:
\[
\mu_\alpha = \langle f(\x), \x^\alpha \rangle.
\]
We observe that, by the biorthogonality relation \eqref{biorthogonal}, the two sides of \eqref{momentprinciple} agree when applied to any polynomial function of $\x$.  However, the convergence of the moment expansion as a distribution depends on the appropriate choice of distribution spaces.  

In this section we first review the theory of asymptotic moment expansions developed by \cite{Estrada}. We then prove Theorem \ref{thm:momentcontinuity} regarding the local uniform convergence of asymptotic moment expansions applied to convolutions.  

\subsection{Asymptotic moment expansions}

We begin by defining the relevant spaces of test functions and distributions.  For $\gamma > 0$, let $\cP_\gamma = \cP_\gamma(\R^d)$ be the space of smooth functions $\psi$ on $\R^d$ with derivatives asymptotically 
bounded by $e^{\gamma |\x|}$, so that
\[
\lim_{|\x| \rightarrow \infty} e^{-\gamma |\x|}\, \psi^{(\alpha)} (\x) = 0, 
\]
for each $\alpha$. 
The topology on $\cP_\gamma$ is generated by the seminorms
\[
||\psi||_{\gamma,\alpha} = \sup_{\x \in \R^d} \left | e^{-\gamma |\x|} \,\psi^{(\alpha)}(\x) \right |,
\]
varying over multiindices $\alpha$. Define the space $\cP=\cP(\R^d)$ by  
\[
\cP = \bigcap_{\gamma>0} \cP_\gamma,
\]
with topology generated by the seminorms $|| \; ||_{\gamma, \alpha}$ as $\gamma$ and $\alpha$ both vary.  $\cP$ is the space of smooth functions with
slower than exponential growth.  The dual spaces to $\cP_\gamma$ and $\cP$ are denoted $\cP_\gamma'$ and $\cP'$ respectively.  Distributions in $\cP_\gamma'$ decay as $e^{-\gamma |\x|}$ or faster, and while those in $\cP'$ have exponential or faster decay.  Clearly $\cP_\gamma' \subset \cP'$ for each $\gamma > 0.$

The asymptotic moment expansion of a distribution $f \in \cP'$ is \cite[Theorem 4.3.1]{Estrada}
\[
f(\sigma \x) \sim \sum_{|\alpha| \geq 0} \frac{(-1)^{|\alpha|}}{\alpha !} \mu_\alpha \; \sigma^{-|\alpha|-d} \;\delta^{(\alpha)}(\x)
\qquad (\sigma \rightarrow \infty),
\]
where 
\[
\mu_\alpha = \langle f(\x), \x^\alpha \rangle
\]
is the $\alpha$th moment of $f$.  This expansion holds in that for any $\psi \in \cP$ and $N \geq 0$,
\begin{equation}
\label{momentexpansion}
\langle f(\sigma \x), \psi(\x) \rangle = 
\sum_{0 \leq |\alpha| \leq N} \frac{\mu_\alpha}{\alpha !} \; \sigma^{-|\alpha|-d} \; \psi^{(\alpha)} (\mathbf{0})
+ \mathcal{O}(\sigma^{-N-d-1}) \qquad (\sigma \rightarrow \infty).
\end{equation}
The above asymptotic expansion is equivalent to the following equation for all $N \geq 0$:
\[
\lims_{\sigma \rightarrow \infty} \sigma^{N+d} \left|\langle f(\sigma \x), \psi(\x) \rangle - 
\sum_{0 \leq |\alpha| \leq N} \frac{\mu_\alpha}{\alpha !} \; \sigma^{-|\alpha|-d} \; \psi^{(\alpha)} (\mathbf{0}) \right| = 0.
\]
Note that for polynomial $\psi$ of degree $\leq N$, the two sides of \eqref{momentexpansion} coincide (without the error term) according to
the biorthogonality relation \eqref{biorthogonal}.  The moment expansion \eqref{momentexpansion} for general $\psi \in \cP$ is an asymptotic version of this biorthogonality relation. 

\subsection{Local uniform convergence of convolved moment expansions}

Here we prove the continuity result, Theorem \ref{thm:momentcontinuity} from Section \ref{sec:intromoments}, which we state here in a more precise form:
\begin{thm4} For all $f \in \cP',\psi \in \cP$, and $N \geq 0$, the $\sigma$-indexed family of functions 
\begin{equation*}
\w \longmapsto \sigma^{N+d} \left( f*\psi_\sigma(\sigma \w)
- \sum_{0 \leq |\alpha| \leq N} \frac{(-1)^{|\alpha|}}{\alpha !} \, \mu_\alpha \; \sigma^{-|\alpha|-d} \; \psi^{(\alpha)} (\w) \right)
\end{equation*}
converges locally uniformly (in $\w$) to the zero function of $\w$ as $\sigma \rightarrow \infty$.
\end{thm4}

Above, ``converges locally uniformly" is shorthand for ``converges uniformly on compact subsets". 
The proof of Theorem \ref{thm:momentcontinuity} is based on that of the asymptotic moment expansion in \cite{Estrada}.  We begin with the following lemma.

\begin{lemma}
\label{rholemma}
Let $\rho=\rho(\w,\y) \in \cP(\R^{2d})$, and for each fixed $\w \in \R^d$ define $\rho_\w(\y)= \rho(\w,\y)$ (hence 
$\rho_\w \in \cP(\R^d)$.)  Suppose that for some integer $N \geq 0$, $\rho$ satisfies
\[
\rho_\w^{(\alpha)}(0) \equiv \frac{\partial^{|\alpha|} \rho}{\partial \y^\alpha} (\w,\y) \bigg|_{\y=\mathbf{0}} = 0
\]
for all $\w$ and each multiindex $\alpha$ with $|\alpha| \leq N$.
Then for any continuous seminorm $||\quad||$ on $\cP(\R^d)$, the following $\sigma$-indexed family of functions of $\w$,
\[
\w \longmapsto \sigma^{N} ||\rho_{\w}(\placehold/\sigma)||,
\]
converges locally uniformly (in $\w$) to the zero function (of $\w$) as $\sigma \rightarrow \infty$.
\end{lemma}

We use the symbol $\placehold$ to denote function or distribution arguments for the purposes the bracket operation $\langle \, , \, \rangle$ or seminorms.  Here, the notation $\rho_{\w}(\placehold/\sigma)$ represents the function mapping $\y \in \R^d$ to $\rho_{\w}(\y/\sigma)$.

\begin{proof}
We prove the stronger statement that the family of functions
\[
\w \longmapsto \lims_{\sigma \rightarrow \infty} \sigma^{N+1} ||\rho_{\w}(\placehold/\sigma)|| 
\]
is locally uniformly bounded in $\w$, where $\lims$ denotes limit superior.
Consider first the seminorm $|| \quad ||_{\gamma,\mathbf{0}}$ for fixed $\gamma>0$, and suppose the lemma is false.  
Then there must be a compact neighborhood $K \subset \R^d$ and a pair of sequences $\{\w_j \in K\}_{j \geq 0}$, $\{\sigma_j \in \R\}_{j \geq 0}$, with $\sigma_j \to \infty$, such that 
\begin{equation}
\label{sup}
\infty = \lim_{j \rightarrow \infty} \sigma_j^{N+1} \left| \left|\rho_{\w_j}(\placehold/\sigma_j)\right | \right|_{\gamma,\mathbf{0}} =
\lim_{j \rightarrow \infty} \sigma_j^{N+1} \sup_{\y \in \R^d} \left|e^{-\gamma|\y|} \rho_{\w_j}(\y/\sigma_j) \right|.
\end{equation}
By passing to a subsequence if necessary, we may assume $\{\w_j\}$ converges to some $\w' \in K$.

Since
\[
\lim_{|\y| \rightarrow \infty} \left|e^{-\gamma|\y|} \rho_{\w_j}(\y/\sigma_j) \right| = 0
\]
for each $j$, the supremum in \eqref{sup} is realized at some $\y_j \in \R^d$.  Hence
\begin{align*}
\infty & =  \lim_{j \rightarrow \infty} \sigma_j^{N+1} \left| \left|\rho_{\w_j}(\y/\sigma_j) \right| \right|_{\gamma,\mathbf{0}}\\
& = \lim_{j \rightarrow \infty} \sigma_j^{N+1}  \left|e^{-\gamma|\y_j|} \rho_{\w_j}(\y_j/\sigma_j) \right|\\
& = \lim_{j \rightarrow \infty} \left( |\y_j|^{N+1} e^{-\gamma|\y_j|(1-\sigma_j^{-1})} \right)
\left( e^{-\gamma|\y_j/\sigma_j|} |\y_j/\sigma_j|^{-N-1} |\rho_{\w_j}(\y_j/\sigma_j)| \right).
\end{align*}
The expression $|\y_j|^{N+1} e^{-\gamma|\y_j|(1-\sigma_j^{-1})}$ is bounded in $j$ since $\sigma_j \rightarrow \infty$, so we must have
\begin{equation}
\label{unbounded}
\lim_{j \rightarrow \infty} e^{-\gamma|\y_j/\sigma_j|} |\y_j/\sigma_j|^{-N-1} |\rho_{\w_j}(\y_j/\sigma_j)| = \infty.
\end{equation}

By passing to a subsequence if necessary,
we may assume that the sequence $\{\y_j/\sigma_j\}$ either approaches the origin as a limit or is bounded away from the origin.  In the first case,  
$\lim_{j \to \infty} \y_j/\sigma_j = \mathbf{0}$, the quantity
\[
|\y_j/\sigma_j|^{-N-1} \left|\rho_{\w'}(\y_j/\sigma_j) \right|
\]
is bounded by the derivative condition on $\rho$, and so the quantity
\[
|\y_j/\sigma_j|^{-N-1} \left|\rho_{\w_j}(\y_j/\sigma_j) \right|
\]
appearing in \eqref{unbounded} is bounded by continuity of the $(N+1)$st derivative of $\rho$.  Therefore 
\[
e^{-\gamma|\y_j/\sigma_j|} |\y_j/\sigma_j|^{-N-1} |\rho_{\w_j}(\y_j/\sigma_j)|
\]
is bounded, contradicting \eqref{unbounded}.  
In the second case, $\liminf_{j \to \infty} |\y_j/\sigma_j|>0$, we note that $|\w_j|$ is bounded since $K$ is compact.  Therefore, the quantity
\[
e^{-\gamma|\y_j/\sigma_j|}|\rho_{\w_j}(\y_j/\sigma_j)| = e^{-\gamma|\y_j/\sigma_j|}|\rho(\w_j, \y_j/\sigma_j)|
\]
appearing in \eqref{unbounded} is less than or equal to
\begin{equation}
\label{tobebounded}
B e^{-\gamma\sqrt{|\w_j|^2+|\y_j/\sigma_j|^2}}|\rho(\w_j, \y_j/\sigma_j)|,
\end{equation}
for some $B>0$.  Quantity \eqref{tobebounded} is bounded in $j$ since 
$\rho(\placehold,\placehold) \in \cP(\R^{2d})$.  Combining this with the boundedness of $|\y_j/\sigma_j|^{-N-1}$ again yields a contradiction of
\eqref{unbounded}.  The lemma is therefore true for the seminorm $|| \quad ||_{\gamma,\mathbf{0}}.$

For the seminorm $||\quad||_{\gamma,\alpha}$ with $|\alpha| >0$ we have
\[
||\rho_\w(\placehold /\sigma)||_{\gamma,\alpha} 
= \sigma^{-|\alpha|} ||\rho_\w^{(\alpha)}(\placehold /\sigma)||_{\gamma,\mathbf{0}},
\]
whereupon we may apply the above argument to $\rho_\w^{(\alpha)}$ in place of $\rho_\w$, yielding the desired result.  Since the family of seminorms 
$|| \quad ||_{\gamma, \alpha}$ generates the topology on $\cP$, the result is true for any continuous seminorm.
\end{proof}

\begin{proof}[Proof of Theorem \ref{thm:momentcontinuity}]
Let 
\[
P_N(\w,\y) = \sum_{0 \leq |\alpha| \leq N} \frac{(-1)^{|\alpha|}}{\alpha !} \y^{\alpha} \, \psi^{(\alpha)}(\w)
\]
be the Taylor expansion of $\psi(\w-\y)$ about $\y = \mathbf{0}$ to order $N$, and define the remainder function $\rho_{N,\w}(\y)$ by 
\[
\rho_{N,\w}(\y) = \psi(\w-\y) - P_N(\w,\y).
\]
Then
\begin{align*}
\nonumber
f*\psi_\sigma(\sigma \w)  & = \sigma^{-d} \big \langle f(\placehold), \psi ((\sigma \w-\placehold)/\sigma ) \big \rangle \\
 & = \sigma^{-d} \big \langle f(\placehold), \psi ( \w-\placehold/\sigma ) \big \rangle \\
& = \sigma^{-d} \big \langle f(\placehold), P_N(\w,\placehold/\sigma)\big \rangle +
\sigma^{-d} \big \langle f(\placehold), \rho_{N,\w}(\placehold/\sigma) \big \rangle\\
& = \sum_{0\leq |\alpha| \leq N}  
\frac{(-1)^{|\alpha|}}{\alpha !} \, \mu_\alpha \; \sigma^{-|\alpha|-d} \; \psi^{(\alpha)} (\w) + 
\sigma^{-d} \big \langle f(\placehold), \rho_{N,\w}(\placehold/\sigma) \big \rangle.
\end{align*}
Rearranging, we obtain
\begin{multline*}
\sigma^{N+d} \left(\langle f(\sigma \placehold), \psi(\w + \placehold) \rangle 
- \sum_{0 \leq |\alpha| \leq N} \frac{\mu_\alpha}{\alpha !} \; \sigma^{-|\alpha|-d} \; \psi^{(\alpha)} (\w) \right)\\
 = \sigma^N \langle f(\placehold), \rho_{N,\w}(\placehold/\sigma) \rangle.
\end{multline*}
To finish, we note that $||\rho_{N,\w}|| = |\langle f(\placehold), \rho_{N,\w}(\placehold) \rangle|$ is a continuous seminorm on $\rho_{N,\w} \in \cP(\R^d)$ for any $f \in \cP'$, and $\rho_N(\w,\y) \equiv \rho_{N,\w} (\y) $ satisfies 
the conditions of Lemma \ref{rholemma}.  Therefore, the family of functions
\[
\w \longmapsto \sigma^N \langle f(\placehold), \rho_{N,\w}(\placehold/\sigma) \rangle
\]
converges locally uniformly in $\w$ to the zero function as $\sigma \rightarrow \infty$, proving the theorem.
\end{proof}


\section{Proof of Unique Determination}
\subsection{General wavelets}
\label{generalproof}

Here we prove our main result, Theorem \ref{generalmarr}, regarding unique determination of a function from the nodes of its wavelet transform at a discrete set of scales.

Let $\psi \in \cP$ be a wavelet, and let $f \in \cP_\gamma' \cap L^1(\R^d)$, for some $\gamma>0$, be a function to be determined.   We define the $\psi$-zeros of $f$ at scale $\sigma > 0$ to be the zeros of $f*\psi_\sigma(\x)$, with $\psi_\sigma(\x) = \sigma^{-d} \psi(\x/ \sigma)$. We recall the genericity condition from the Introduction:

\begin{generic}
The regular zero set of any derivative of fixed order $n$ is not contained in the zero set of any other derivative of fixed order $m$, for any $n,m \geq 0$.
\end{generic}

Above, the \emph{regular} (or \emph{transverse}) zeros of $\psi$ are those those around which the function takes both positive and negative values in any open neighborhood.  A \emph{derivative of fixed order $n$} (or an \emph{order-$n$ derivative}) of $\psi$ is a function of the form $\sum_{|\alpha|=n} C_\alpha \psi^{(\alpha)}$ for some $n \geq 0$, where the $C_\alpha$ are constants not all equal to zero, defined up to multiplication by a nonzero scalar.

Our main result can now be stated as follows: 

\begin{thm1}
Let $f \in \cP_\gamma' \cap L^1(\R^d)$ (for any $\gamma >0$) and $\psi \in \cP$ satisfy the genericity condition.  Then $f$ is uniquely determined (up to a constant multiple) by its $\psi$-edges at any sequence
of positive scales $\{\sigma_j\}_{j=1}^\infty$ tending to infinity.
\end{thm1}

\begin{proof} For convenience we introduce $\w = \x/\sigma$.  Moment expansion (Theorem \ref{thm:momentcontinuity}) gives
\begin{equation}
\label{prezceqn}
f*\psi_\sigma(\x) 
 \sim \sum_{|\alpha| \geq 0} \frac{(-1)^{|\alpha|}}{\alpha !} \; \mu_\alpha \; \sigma^{-|\alpha|-d} \; \psi^{(\alpha)} (\w) 
\qquad (\sigma \rightarrow \infty).
\end{equation}
Also for convenience, we introduce the function $Z(\sigma, \w) = \sigma^{n_0 +d} \;  (f*\psi_\sigma) (\sigma \w)$, where $n_0$ is 
the order of the lowest-order nonzero moment of $f$.  $Z$ admits the moment expansion
\begin{equation}
\label{zceqn}
Z(\sigma, \w) \sim \sum_{|\alpha| \geq n_0} \frac{(-1)^{|\alpha|}}{\alpha !} \; \mu_\alpha \; \sigma^{n_0-|\alpha|} \; 
 \psi^{(\alpha)} (\w) \qquad (\sigma \rightarrow \infty).
\end{equation}
By locally uniform convergence of the moment expansion (Theorem \ref{thm:momentcontinuity}, in the case $N=n_0$), as $\sigma \rightarrow \infty$, $Z(\sigma,\w)$ converges locally uniformly in $\w$ to
\begin{equation}
\label{asedge}
z(\w) = (-1)^{n_0} \sum_{|\alpha| = n_0} \frac{\mu_\alpha}{\alpha !} \; \psi^{(\alpha)} (\w).
\end{equation}

The $\psi$-zeros at scale $\sigma$ correspond to the zeros, in $\w$, of $ Z(\sigma, \w)$.  By assumption, we are given the zero sets $E_j = \{\w :Z(\sigma_j, \w)=0\} \subset \R^d$ at scale $\sigma=\sigma_j$ for each $j \geq 0$.  We call the 
limiting set $E$ of $\{E_j\}$ as $j \rightarrow \infty$ (i.e. the set of all limits of sequences  $\{\w_j \in E_j\}_{j \geq 0}$) the \emph{asymptotic zero set}.  

$E$ contains all regular zeros of $z$ since regular zeros persist under small locally uniform perturbations. 
So if $\w'$ is a regular zero of $z(\w)$ we may, from knowledge of $\{E_j\}_{j \geq 0}$, choose a sequence $\{\w_j \in E_j\}_{j \geq 0}$ such that
$\lim_{j\rightarrow \infty} \w_j =\w'.$
By locally uniform convergence (Theorem \ref{thm:momentcontinuity}), we may substitute $\sigma = \sigma_j$ and $\w = \w_j$ into \eqref{zceqn}, obtaining
\[
0 = Z(\sigma_j, \w_j) \sim \sum_{|\alpha| \geq n_0} \frac{(-1)^{|\alpha|}}{\alpha !} \; \mu_\alpha \; \sigma_j^{n_0-|\alpha|} 
\; \psi^{(\alpha)} (\w_j).
\]
This expansion holds in the sense that for each $k \geq 0$, the partial sum of the right-hand side with $|\alpha|$ up to $n_0+k$ vanishes to order $\sigma_j^{-k}$ as
$j \rightarrow \infty$:
\begin{equation}
\label{prerecur}
\lim_{j \rightarrow \infty} \sigma_j^k \sum_{n_0 \leq |\alpha| \leq n_0+k } \frac{(-1)^{|\alpha|}}{\alpha !} \; \mu_\alpha \; 
\sigma_j^{n_0-|\alpha|} \; \psi^{(\alpha)} (\w_j) =0,
\end{equation}
for all $k \geq 0$.  We separate the left-hand side of \eqref{prerecur} into terms involving moments of order $n_0+k$ and those involving lower-order moments:
\begin{multline}
\label{recursion} 
\sum_{|\alpha|=n_0 + k}  \frac{(-1)^{|\alpha|}}{\alpha !} \; \mu_\alpha \psi^{(\alpha)} (\w')\\
+ \lim_{j \rightarrow \infty} \sum_{n_0 \leq |\alpha| < n_0 + k} \frac{(-1)^{|\alpha|}}{\alpha !} \; \mu_\alpha 
\; \sigma_j^{n_0+k-|\alpha|} \; \psi^{(\alpha)} (\w_j) =0.
\end{multline}
 
The two terms of Equation \eqref{recursion} form a linear recursion relation for the moments $\mu_\alpha$ of order $|\alpha|=n_0+k$ in terms of lower-order moments. We now show by induction on $k$ that Equation \eqref{recursion} recursively determines all moments of $f$ up to a constant multiple. 

As a basis step we observe that, for $k=0$, the second term on the left-hand size of Equation \eqref{recursion} vanishes, while the first term is equal to $z(\w')$.  Since $Z(\sigma,\w)$ converges locally uniformly in $\w$ to $z(\w)$, the asymptotic edge $E$ contains the regular zeros set of $z$ and is contained in the zero set of $z$. Furthermore, since $z$ is an order-$n_0$ derivative of $\psi$, the genericity condition ensures that $E$ cannot contain the regular zero set of any other fixed-order derivative of $\psi$ (if so, this regular zero set would also be contained in the zero set of $z$, violating the genericity condition).  Thus $z(\w)$ is the unique fixed-order derivative of $\psi$ whose regular zeros are contained in $E$. Since $E$ is uniquely determined by the given zero sets $\{E_j\}$, it follows that $n_0$ and all moments $\mu_\alpha$ with $|\alpha|=n_0$, up to a common multiple, are uniquely determined by these zero sets.

Now assume for induction that for some $k>0$, all moments $\mu_\alpha$ with $|\alpha|<n_0+k$ are known.  Then the first term on the left-hand side of Equation \eqref{recursion} can be evaluated at any $\w' \in E$ by choosing a corresponding sequence $\{\w_j \in E_j\}$ with $\w_j \to \w'$ and evaluating the second term.  The genericity condition ensures that the moments $\mu_\alpha$ with $|\alpha|=n_0+k$ are uniquely determined by the values of the first term as $\w'$ ranges over the regular zeros of $z$, since the difference between any two distinct solutions would be an order-$(n_0 + k)$ derivative of $\psi$ that is identically zero on the regular zero set of $z$, an order-$n_0$ derivative of $\psi$.  Thus the moments of order $n_0+k$ are uniquely determined by the lower-order moments together with the given zero sets $\{E_j\}$.  If the lower-order moments are known only up to a common multiple, then since Equation \eqref{recursion} is linear in the moments, those of order $n_0+k$ are determined up to this same common multiple.  This completes the induction, showing that all moments of $f$ are determined up to a constant multiple by the zero sets $\{E_j\}$.

To determine $f$ from its moments we can use the Fourier transform
\[
\hat{f}(\omega) = \int_{\R^d} f(\x) e^{-i \omega \cdot \x} \; d \x.
\]
We claim that $\hat{f}(\omega)$ is well-defined and analytic for all $\omega \in \C^d$ with 
$|\Im \omega|<\gamma$.  This result is well-known as a version of the Payley-Weiner theorem for $f \in \cP_\gamma' \cap L^2$; we provide the argument for $f \in \cP_\gamma' \cap L^1$.  

Fix such an $\omega$.
The Fourier transform $\hat{f}$ is well-defined at $\omega$ since $f \in \cP_\gamma' \cap L^1(\R^d)$.
Furthermore, the (complex) partial derivative of $\hat{f}$ in the $j$th coordinate at $\omega$ is given by
\begin{equation}
\label{fourierderivative}
\frac{\partial \hat{f}}{\partial \omega_j} (\omega) = \lim_{\substack{\epsilon \rightarrow 0\\ \epsilon \in \C}} 
\int_{\R^d} f(\x) \frac{1}{\epsilon}
\left(e^{-i (\omega \cdot \x+\epsilon x_j)}- e^{-i \omega \cdot \x}\right) \; d \x.
\end{equation}
Fix $\lambda \in \R$ satisfying $\Im \omega < \lambda< \gamma$.  
For sufficiently small $\epsilon$, the integrand in \eqref{fourierderivative} is absolutely bounded over all $\x \in \R^d$ by
\[
|\x f(\x)| e^{ \lambda |\x|},
\]
which is integrable since $f \in \cP_\gamma' \cap L^1(\R^d)$.  By dominated convergence, the limit and integral in \eqref{fourierderivative} can be interchanged, yielding
\[
\frac{\partial \hat{f}}{\partial \omega_j} (\omega) = \int_{\R^d} (-ix_j) f(\x) e^{-i \omega \cdot \x} \; d \x.
\]
This shows that all complex first partials of $\hat{f}(\omega)$ exist; thus $\hat{f}$ is analytic at $\omega$.  

Using dominated convergence to iteratively evaluate derivatives of $\hat{f}(\omega)$ as in \eqref{fourierderivative}, we obtain the Taylor expansion
\[
\hat{f}(\omega) = \sum_{|\alpha| \geq 0|}  \mu_\alpha \frac{(-i\omega)^\alpha}{\alpha !}.
\]
By analytic continuation, the moments $\{ \mu_\alpha \}_{|\alpha \geq 0}$ uniquely determine $\hat{f}$ on $\R^d$. Since the Fourier transform is one-to-one on $L^1(\R^d)$, $f$ is uniquely determined
by its moments.  This completes the proof.
\end{proof}

This theorem can be described as the strongest of its kind in two senses.  First (Section \ref{necessity}), the theorem is false if the exponential decay required by the condition $f\in \cP_\gamma$ is relaxed to algebraic decay.  

Second, if the above genericity condition \ref{generic} is relaxed even mildly, the theorem becomes false.  Indeed, if we relax the condition to require that the \emph{regular} (instead of all) zeroes of any fixed order derivative $\psi^{(n)}$ never contain the regular zeroes of any other fixed order derivative $\psi^{(n)}$, the theorem is false.  

The simplest example of this involves on $\R^1$ a wavelet based on the Gaussian, for which a one term and a two term Dirac delta initial condition have the same zeroes.

This wavelet $\psi(x)$ is based on a triplet of reals $(a,b,x^*)$ solving the system of equations
\begin{subequations}
\label{eq:countersystem}
\begin{align}
\label{eq:countersystema}
-x^*e^{-(x^*)^2/2}+ax^*+b &=0\\
\label{eq:countersystemb}
\left({(x^*)}^2-1 \right) e^{-(x^*)^2/2}+a &=0\\
\label{eq:countersystemc}
\sqrt{3}e^{-3/2}-\sqrt{3}a+b &=0,
\end{align}
\end{subequations}
where $x^* \neq -\sqrt{3}$.  A numerical solution exists with $x^* \approx 0.71$, $a \approx 0.38$, and $b \approx 0.28$.  Note that $x^*$ is a transcendental (non-algebraic) number, as we will show.

Now let
\[
\psi(x) = -xe^{-x^2/2}+ax+b.
\]
$\psi(x)$ has a non-regular zero at $x=x^*$ and a regular zero at $x=-\sqrt{3}$. Its derivative $\psi'(x)$ has zeros at $x=\pm x^*$ Its second derivative $\psi''(x)$ has regular zeros at $x=0$ and $x= \pm\sqrt{3}$ and no non-regular zeros (see Figure 1).  All derivatives of order $n \geq 3$ have (algebraic) zeros given by the roots of the Hermite polynomial $H_{n+1}(x)$. 

\begin{figure}
\begin{center}
\includegraphics[scale=0.8]{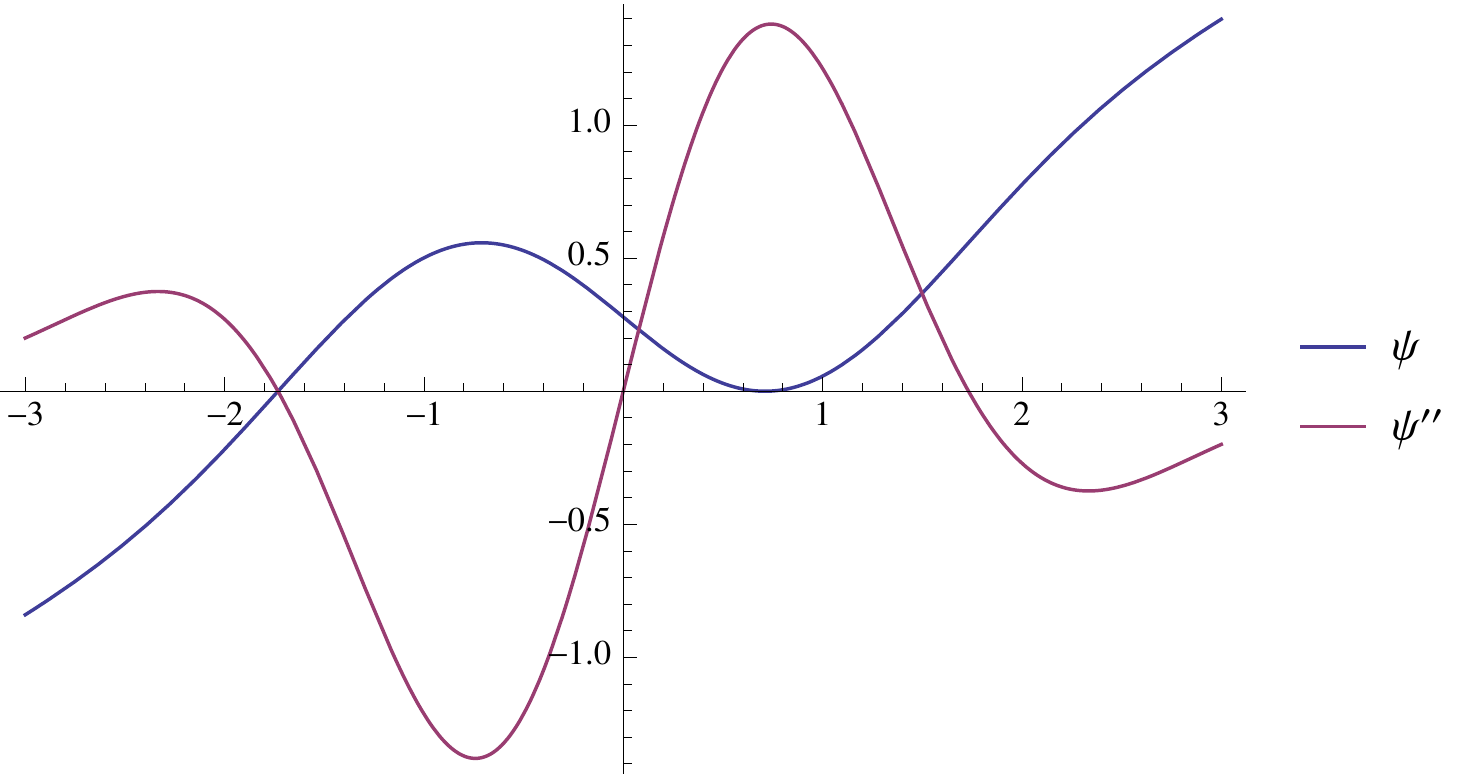}
\caption{Graphs of $\psi(x)$ and $\psi''(x)$. Note $\psi(x)$ has a non-regular zero at $x=x^*$ and a regular zero at $x=-\sqrt{3}$, while $\psi''(x)$ has regular zeros at $x=0$ and $x= \pm\sqrt{3}$ and no non-regular zeros}
\end{center}
\end{figure}

The wavelet $\psi$ satisfies the weakened genericity condition.  Indeed, the zeros of $\psi$ are clearly not contained in those of $\psi'$ or $\psi''$. Further, the zeroes of $\psi^{(n)}$ for $n>2$ are zeroes of Hermite polynomials, which are algebraic, so they cannot contain $x^*$.  Similarly, $\psi'$ has zeroes at $\pm x^*$, which are not both contained in the zeroes of $\psi$, nor in the (algebraic) zeroes of $\psi^{(n)}$ for $n \geq 2$.  The zeroes of $\psi^{(n)}$, $n\geq 2$, are algebraic and so cannot be contained in those of $\psi$ or $\psi'$, and are not contained in the zeroes of any other $\psi^{(m)}$ ($m\geq 2$), since the zeroes of different Hermites are never contained in each other (see Section \ref{gaussianproof}).

Note however that $\psi$ does \emph{not} satisfy the original genericity condition, since its only regular zero (at $x=-\sqrt{3}$) is contained in the zero set of $\psi''$.  

For $c>0$ sufficiently small and $\sigma>0$ sufficiently large, the functions $\psi(w)$ and $\psi(w) + c\sigma^{-2}\psi''(w)$ both have zeros only at $x=-\sqrt{3}$.  Thus the initial distributions $\delta^{(0)}$ and $\delta^{(0)} + c\delta^{(2)}$ cannot be distinguished by their zeros when the scaling $\sigma$ is sufficiently large.  This shows that the weakened genericity condition is insufficient for Marr's conjecture to hold. 

It remains to show that $x^*$ is transcendental (i.e., not an algebraic number).  To this end, solving \eqref{eq:countersystemb} for $a$ we have
\[
a= \left(1-(x^*)^2 \right) e^{(-x^*)^2/2}.  
\]
Now substituting for $a$ in \eqref{eq:countersystema} and solving for $b$:
\[
b = (x^*)^3 e^{(-x^*)^2/2}.  
\]
Now substituting in \eqref{eq:countersystemc} and rearranging,
\[
\sqrt{3}e^{-3/2+(x^*)^2/2}=\sqrt{3} \left(1-(x^*)^2 \right)-(x^*)^3.
\]
If $x^*$ were algebraic then the left side would be transcendental (since the exponentials of non-zero algebraics are transcendental by the Hermite-Lindemann theorem) while the right side would be algebraic, which would give a contradiction.  Therefore $x^*$ is transcendental.


\subsection{Ricker wavelets and the Marr conjecture}
\label{gaussianproof}

We now specialize Theorem \ref{generalmarr} to the Ricker (Mexican hat) wavelet $M(\x)=\Delta G(\x)$, which is clearly in $\cP$. To apply the theorem we need to verify Conditions 1 and 2.

In one dimension, the Ricker wavelet has derivatives
\begin{equation}
\label{Hermite}
M^{(n)}(x) = (-1)^n H_{n+2}(x) G(x),
\end{equation}
where 
\begin{equation}
\label{explicitHermite}
H_n(x) = \sum_{k=0}^{\lfloor n/2 \rfloor} (-1)^k \frac{n!}{k!(n-2k)!2^k} x^{n-2k}
\end{equation}
is the $n$th Hermite polynomial.  We invoke a theorem of Schur \cite{schur} that $H_{2m}(x)$ 
and $H_{2m+1}(x)/x$ are irreducible (cannot be factored) over the rationals for all $m \geq 0$.   Two distinct irreducible monic polynomials over the rationals cannot have a 
common real root $x_0$, or else they would both be divisible by the minimal polynomial of $x_0$ (i.e. the unique rational monic polynomial of minimal degree that has a root 
at $x_0$; see \cite[Theorem V.1.6]{Hungerford}).  Thus any two distinct Hermite polynomials have at most the root $x=0$ in common.  Furthermore, the relation
\[
H_n'(x) = nH_{n-1}(x),
\]
together with the above irreducibility result, implies that Hermite polynomials have no multiple roots, i.e. all zeros are regular. The genericity condition on $M(x)$ follows, proving Corollary \ref{cor:OneGaussian}(a) in the case $\sigma_i \to\infty$:

\begin{cor2a} \emph{(Infinite limit case)}
Any $f \in \cP_\gamma' \cap L^1(\R)$ is uniquely determined (up to a constant multiple) by its Gaussian edges at any sequence of scales tending to infinity.
\end{cor2a}

In higher dimensions, the genericity condition on $M$ reduces to polynomial relations.  Partial derivatives of $M$ are described by the Laplace-Hermite polynomials $L_\alpha(\x)$, defined in equation \eqref{eq:LHrelation}, which have the explicit form
\[
L_\alpha(\x) = \sum_{i=1}^d \left ( H_{\alpha_i+2}(x_i) \prod_{\substack{1 \leq j \leq d\\j \neq i}} H_{\alpha_j}(x_j)
\right ).
\]
The genericity condition reduces to statements about zeroes of these polynomials, as stated in Corollary \ref{cor:dGaussian} above.  We have numerically verified this condition for the case $d=2$, $n=0$, $m \leq 15$.


\section{Geometry of Gaussian Edge Contours}
\label{edgesec}

Having proven the Marr conjecture in one dimension (Corollary \ref{cor:OneGaussian}(a)), in the remainder of this work we ask whether this result can be extended to other sequences of scales and to functions that decay less rapidly than those in $\cP'_\gamma$.  We therefore restrict our focus to one-dimensional Gaussian edges---that is, zeros of $f*M_\sigma$, or equivalently, of $\Delta (f * G_\sigma)$---for $f \in L^1(\R)$.  Our results are summarized in Corollary \ref{cor:OneGaussian}(b,c) above.  

To start, we give a characterization of the geometry of one-dimensional Gaussian edges, which we will later use in proving unique determination from sequences of bounded-scale edges.  Since these edges are nodes of a heat equation solution, we will represent scale using the variable $t=\sigma^2$ rather than $\sigma$.  

Given $f \in \cP_\gamma' \cap L^1(\R)$, define
\[
F(x,t) = f * G_{\sqrt{t}} (x).
\]
$F$ is jointly analytic in both variables on the upper half-plane $H_+ = \R \times (0, \infty)$ (e.g. \cite[Theorem 10.3.1]{Cannon}). 
Both $F$ and $F_{xx}=\frac{\partial^2 F}{\partial x^2}$ satisfy the heat equation \eqref{heateqn}, and so are  subject to the following maximum principle  (e.g. \cite[Theorem 15.3.1]{Cannon}):

\begin{theorem}
\label{maxprinciple}
For $t_2>t_1 \geq 0$, let $s_1,s_2:[t_1,t_2]\rightarrow \R$ be continuous functions with $s_1(t)<s_2(t)$ for all $t \in (t_1, t_2]$.  Let $D$ be the parabolic
interior
\[
D=\{(x,t): t \in (t_1, t_2], s_1(t)<x<s_2(t) \},
\]
Let $u(x,t)$ satisfy the heat equation $u_{xx} = \frac{1}{2}u_t$ on $\bar{D}$, the closure of $D$.  
Then if the maximum (or minimum) value of $u$ over $\bar{D}$ is achieved on $D$, $u$ is constant on $\bar{D}$.
\end{theorem}

An immediate consequence of the maximum principle is that $F_{xx}$ has no isolated zeros, since such a zero would be a local extremum.  Furthermore, since $F_{xx}$ is analytic,
the resolution of analytic singularities in two real dimensions 
\cite{Hironaka2, Hironaka, Bierstone} (or Puiseux series expansion, e.g.~\cite{Bierstone2} or Theorem 4.2.11 of \cite{Krantz}) implies that for each zero $(x_0, t_0) \in H_+$ of $F_{xx}$ 
there must be a neighborhood $W \subset H_+$ containing $(x_0, t_0)$ and a collection of injective real-analytic mappings
\begin{align}
\nonumber
(\chi_1, \tau_1) & : (-\epsilon_1, \epsilon_1) \rightarrow W,\\
\nonumber
& \vdots\\
\label{parameterization}
(\chi_k, \tau_k) & : (-\epsilon_k, \epsilon_k) \rightarrow W,
\end{align}
the images of which intersect only at $(x_0, t_0) = (\chi_1(0), \tau_1(0)) = \ldots = (\chi_k(0), \tau_k(0))$, and the union of whose images is precisely the 
zero set of $F_{xx}$ in $W$.  By analytic continuation of these mappings, the zero set of $F_{xx}$ in $H_+$ can be uniquely described as a union of real-analytic curves or
curve segments that have locally injective parameterizations of the form \eqref{parameterization} around each point, endpoints (if they exist) only on the line 
$t=0$, and whose intersections form a discrete subset of $H_+$.  We call these curves and curve segments \emph{edge contours}.

It is commonly observed computationally \cite{Yuille} that edge contours either form arcs from one point on the $x$-axis to another, or else 
extend from $t=0$ to $t=\infty$.  Solutions for which new edge contours are generated with increasing $t$ have not been observed numerically or analytically.  This observation has been formalized and proven in several ways \cite{Babaud, Hummel}; here we prove Theorem \ref{subset}, which strengthens previous formalizations. We begin with a lemma.

\begin{lemma}
\label{nocreate1}
If $F_{xx}(x_0, t_0) = 0$, then $F_{xx}$ has at least one zero in any rectangle $[x_0-\epsilon, x_0+\epsilon] \times [t_0 - \delta, t_0)$ with $\epsilon, \delta>0$.
\end{lemma}

\begin{proof}
Note that $F_{xx}$ satisfies the heat equation $\left(F_{xx}\right)_{xx} = \frac{1}{2} \left(F_{xx}\right)_t$ and is therefore subject to the maximum 
principle.  Assume the theorem is false, that is, there is a rectangle $[x_0-\epsilon, x_0+\epsilon] \times [t_0 - \delta, t_0)$ that contains no
zeros of $F_{xx}$.  Taking $t_1=t_0-\delta$, $t_2 = t_0$, $s_1(t) = x_0-\epsilon$, $s_2(t) = x_0+\epsilon$, and defining $D$ accordingly as in the statement of Theorem 6,
we find that $F_{xx}$ is either maximized or minimized
over $\bar{D} = [x_0-\epsilon, x_0+\epsilon] \times [t_0 - \delta, t_0]$ at $(x_0,t_0) \in D$ and hence $F_{xx} \equiv 0$ on $\bar{D}$, contradicting our assumption.
\end{proof}

As an immediate consequence, edge contours cannot have local minima in $t$:

\begin{corollary}
\label{nominimum}
If $(\chi, \tau):(-\epsilon, \epsilon) \rightarrow H_+$ is a local parameterization of an edge contour of $F$, then $\tau$ has no local minimum on $(-\epsilon, \epsilon)$.
\end{corollary}

We next define a \emph{persistent} edge contour as an edge contour that extends to arbitrarily large values of $t$ (or equivalently, arbitrarily large values of $\sigma = \sqrt{t}$).  We immediately obtain the following two results:

\begin{corollary}
\label{onetoone}
A persistent edge contour can intersect the line $t=t_1$ no more than once.
\end{corollary}

\begin{corollary}
\label{nomaximum}
If $(\chi, \tau):(-\epsilon, \epsilon) \rightarrow H_+$ is a local parameterization of a persistent edge contour of $F$, then $\tau$ has no local maximum on $(-\epsilon, \epsilon)$.
\end{corollary}

\begin{proof}
Topologically, for each local maximum of a curve that is not a global maximum, there must be also be a local minimum.  Since persistent edge contours have no global maxima or local minima, they therefore cannot have local maxima.
\end{proof}

We now prove Theorem \ref{subset}, which formalizes the observation that edge contours are not generated with increasing $t$:

\begin{edgethm}
For $0<t_1 < t_2$, the edge contours of $f \in \cP'_\gamma \cap L^1(\R)$ intersecting the line $t=t_2$ are a subset of those that intersect $t=t_1$.
\end{edgethm}

\begin{proof}
Since edge contours are never minimized in $t$ (Corollary \ref{nominimum}), it only needs to be shown that there is no edge contour whose $t$-value decreases asymptotically to an intermediate value $t'$, $t_1 < t' < t_2$, as its $x$-value diverges 
to positive or negative infinity.

Assume the contrary, and without loss of generality assume the $x$-value of the edge contour in question diverges to positive infinity as $t$ decreases to $t'$.
Then there is a locally analytic curve $s \mapsto (\chi(s), \tau(s))$ defined for all $s$ greater than some $s_0$, with 
\begin{itemize}
\item $F_{xx}(\chi(s), \tau(s)) = 0$, $\forall s > s_0$,
\item $\tau(s)$ monotonically decreasing,
\item $\lim_{s \rightarrow \infty} \chi(s) = \infty$,
\item $\lim_{s \rightarrow \infty} \tau(s) = t'$.
\end{itemize}

For $x_1>\chi(s_0)$, define $D_1 \subset H_+$ to be the closed connected region bounded by the curve $(\chi(s), \tau(s))$ and the lines $t=t_2$ and $x=x_1$ (Figure 
\ref{proofpic}).  Since $F_{xx}$ is zero on the curve $(\chi(s), \tau(s))$, the maximum principle (Theorem \ref{maxprinciple}) implies that $|F_{xx}(x,t)|$ achieves its maximum value over $D_1$ at a point on the line $x=x_1$.  We denote this maximizing point $(x_1, t_1^*)$.  
For any $x_2>x_1$, $|F_{xx}|$ achieves its maximum over the domain $D_2$ (defined similarly to $D_1$ with $x_1$ replaced by $x_2$; see Figure \ref{proofpic}) at a point on the line $x=x_2$,
 which we denote $(x_2, t_2^*)$.
Moreover, $|F_{xx}(x_2,t_2^*)|> |F_{xx}(x_1,t_1^*)|$ since $D_1 \subset D_2$.  Iterating this argument, we obtain a sequence $\{(x_i, t^*_i)\}_{i \geq 1}$, with 
$x_i \rightarrow \infty$ and $|F_{xx}(x_i, t^*_i)|$ monotonically increasing.  Thus 
\begin{equation}
\label{Fxxnonzero}
\lim_{i \to \infty} |F_{xx}(x_i, t^*_i)| > 0.
\end{equation}

\begin{figure}
\begin{center}
\includegraphics{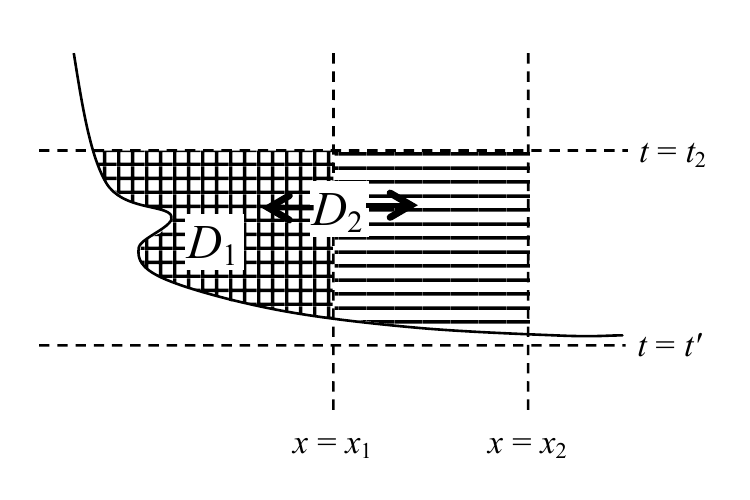}
\caption{Regions defined in proof of Theorem \ref{subset}.  $D_1$ is bounded by the curve $(\chi(s), \tau(s))$ and the lines $t=t_2$ and $x=x_1$.  $D_2$ is bounded by the curve $(\chi(s), \tau(s))$ and the lines $t=t_2$ and $x=x_2$, with $x_2 > x_1$.}
\label{proofpic}
\end{center}
\end{figure}

Since $x_i \to \infty$ while $t_i^*$ is confined to the interval $(t', t_2]$ for each $i$, it is easily verified that the sequence of functions
\[
\left \{ G_{\sqrt{t_i^*}}^{(2)} (x_i - \placehold) \right \}_{i \geq 1}
\]
converges as $i \to \infty$ to the zero function in the topology of $\cP_\gamma$.  (Above, $\placehold$ represents the function argument of elements of $\cP_\gamma$.) Then, since $f \in \cP_\gamma'$,
\[
 \lim_{i \to \infty} F_{xx}(x_i,t_i^*) = \lim_{i \to \infty} f * G_{\sqrt{t_i^*}}^{(2)} (x_i)
 =  \lim_{i \to \infty}  \left \langle f(\placehold), G_{\sqrt{t_i^*}}^{(2)} (x_i - \placehold) \right \rangle = 0.
\]
This contradicts \eqref{Fxxnonzero}; hence no such edge contour exists.
\end{proof}

\section{Reconstruction From Other Edge Sequences}
\label{finiteedges}

The above proof of unique determination from Gaussian edges (Marr's conjecture; Corollary \ref{cor:OneGaussian}(a)) uses only the asymptotics of the edges of $f$ for large scales $\sigma$.  This is unexpected, 
since one would anticipate more information would arise from small-scale rather than large-scale edges.
We show here that in the one-dimensional Gaussian case, a sequence of bounded-scale edges (i.e.~with $\sigma$ remaining bounded) is also sufficient to uniquely determine any
$f \in \cP_\gamma' \cap L^1(\R)$, as long as the sequence of scales has a positive limit point.

\subsection{Sequences of Scales With a Positive Limit Point}

With $f$ and $F$ as above, let $\{t_j\}_{j \geq 1}$ be a sequence of positive real numbers with a limit point $t' >0$, for which the solutions 
(in $x$) to $F_{xx}(x, t_j)=0$ are given.  

The asymptotic edge (Section \ref{generalproof}) for the one-dimensional Ricker wavelet is given by the zeros of $H_{n_0+2} \left(x/\sqrt{t} \right)$, where $n_0$ is the order of the first nonzero moment of $f$.  Since $H_{n_0+2}$ has $n_0+2$ distinct regular real roots, there are exactly $n_0 + 2$ persistent edge contours.  Theorem \ref{subset} implies that the persistent edge contours intersect the lines $t=t_j$ for all $j$, as well as the limiting line $t=t'$.  Further, by Lemma \ref{nocreate1}, the persistent edge contours cross the lines $t=t_j$ rather than achieving local minima at the intersection points. Thus by analytic continuation, the persistent edge contours are uniquely determined by the given solutions to $F_{xx}(x,t_j)=0$.  The infinite-limit case of Corollary \ref{cor:OneGaussian}(a) guarantees that the persistent edge contours uniquely determine $f$.  We have thus proved:

\begin{cor2a} \emph{(General case)}
Any $f \in \cP_\gamma' \cap L^1(\R)$ is uniquely determined, up to a constant multiple, by its Gaussian edges at any set of scales with at least one limit point in $(0, \infty]$.
\end{cor2a}

\subsection{Sequences of Scales Converging Only to Zero}

Perhaps surprisingly, unique determination is not guaranteed for a set of scales whose only limit point is zero, as stated in Corollary \ref{cor:OneGaussian}(b).  To demonstrate this, we construct a compactly supported  $h \in  L^1(\R)$ with the property that the Gaussian edges of $G(x)=(2\pi)^{-1/2}e^{-x^2/2}$ and $G(x)+h(x)$ agree on an infinite sequence of scales $\{ \sigma_k \} \rightarrow 0$.  The function $h$ is defined by its second derivative $\Delta h \in \cP'$, which we represent as an infinite sum
\[
\Delta h = \sum_{n=1}^\infty c_n J_{\alpha_n, \beta_n}.
\] 
Above, for any real numbers $0 < |\beta| < \alpha < 1$, the distribution $J_{\alpha, \beta} \in \cP'$ is defined as a combination of point masses located at $x=\pm(1+\beta) \pm \alpha$:
\begin{align*}
J_{\alpha,\beta}(x) & =  \delta(x + 1 + \alpha + \beta) - \delta(x + 1 - \alpha + \beta) \\
& \quad - \delta(x-1 + \alpha - \beta) + \delta(x-1 - \alpha - \beta).
\end{align*}
We will choose $c_n$, $\alpha_n$, and $\beta_n$ inductively, so that the edge contours of $G+h$ oscillate about those of $G$ as
$\sigma \rightarrow 0$.  

We begin by setting $c_1 =1$ and choosing $0< -\beta_1<\alpha_1<1$ arbitrarily. We define $h_1$ by 
\[
\Delta h_1(x) = c_1 J_{\alpha_1, \beta_1}(x),
\]
together with the requirement that $h_1$ be compactly supported.  (We assume all $h_n$ are compactly supported, and so can
be defined by their second derivatives.)  The function $h_1$ is illustrated in Figure \ref{Jfig}.

There are two edge contours of $G$ (i.e. zero curves of $\Delta (G*G_\sigma)=G^{(2)}_{\sqrt{\sigma^2+1}}$), described by $x= \pm \sqrt{\sigma^2+1}$.
Since $J_{\alpha_1, \beta_1}$ is nonnegative/nonpositive wherever $G^{(2)}$ is, the addition of $h_1$ to $G$ creates no new 
edge contours (any such created edge contours would have to manifest themselves at 
arbitrarily small scales by Theorem \ref{subset}),  
and perturbs the edge contours of $G$ symmetrically about the $\sigma$-axis (by the symmetry of $J_{\alpha_1, \beta_1}$).  

Furthermore, since $\beta_1<0$, the positive point masses of $J_{\alpha_1, \beta_1}$ are closer to $\pm 1$ than the negative ones.  Thus there is a sufficiently small $\sigma_1>0$, 
for which
\[
\Delta (h_1 * G_{\sigma_1} ) \left(\pm \sqrt{\sigma_1^2 + 1} \right) > 0.
\]

Now suppose inductively that for some $n \geq 1$ we have a compactly supported function $h_n \in L^1(\R)$ such that $\Delta h_n$ is zero in neighborhoods of $\pm 1$,
and a strictly decreasing sequence of positive scales $\{ \sigma_1, \ldots \sigma_n \}$ such that
\begin{equation}
\label{zeroscaleinduct}
\begin{cases}
\Delta (h_n * G_{\sigma_k}) \left(\pm \sqrt{\sigma_k^2 + 1} \right) > 0 & \text{for $k$ odd}\\
\Delta (h_n * G_{\sigma_k}) \left(\pm \sqrt{\sigma_k^2 + 1} \right) < 0 & \text{for $k$ even,} \qquad 1 \leq k \leq n.
\end{cases}
\end{equation}

As an induction step, we will choose real numbers $c_{n+1}>0$, $0 < (-1)^{n+1} \beta_{n+1} < \alpha_{n+1}<1$, and $0 < \sigma_{n+1} < \sigma_n$,
such that if $h_{n+1} \in L^1(\R)$ is defined by
\begin{align*}
\Delta h_{n+1}(x) = & \Delta h_n(x) + c_{n+1} J_{\alpha_{n+1}, \beta_{n+1}},
\end{align*}
then 
\begin{equation}
\label{zeroscaleinduct2}
\begin{cases}
\Delta (h_{n+1} * G_{\sigma_k}) \left(\pm \sqrt{\sigma_k^2 + 1} \right) > 0 & \text{for $k$ odd}\\
\Delta (h_{n+1} * G_{\sigma_k}) \left(\pm \sqrt{\sigma_k^2 + 1} \right) < 0 & \text{for $k$ even,} \qquad 1 \leq k \leq n+1.
\end{cases}
\end{equation}

\begin{figure}
\begin{center}
\includegraphics{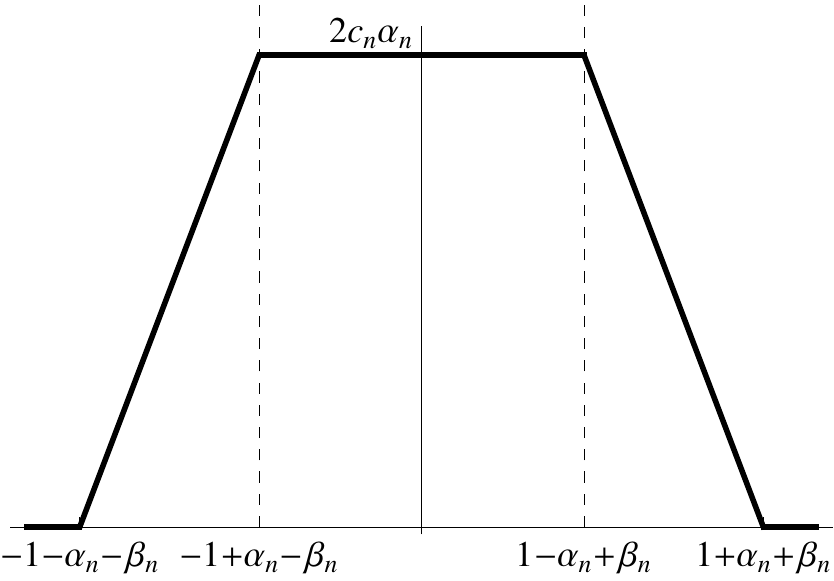}
\caption{The graph of $h_{n}-h_{n-1}$, which satsifies $\Delta (h_{n}-h_{n-1}) = c_n J_{\alpha_n, \beta_n}$.  Note that $\displaystyle \int_{-\infty}^\infty | h_{n}(x)-h_{n-1}(x)| \; dx = 4c_n \alpha_n (1+\beta_n)$.  For $n=1$ this is the graph of $h_1$.  $h$ is constructed as an infinite sum of functions of this form.}
\label{Jfig}
\end{center}
\end{figure}

First, note that for any fixed $\sigma$ (in particular for $\sigma=\sigma_k$ with $1 \leq k \leq n$), the quantity
\[
\left | J_{\alpha_{n+1}, \beta_{n+1}} * G_\sigma (x) \right|
\]
is uniformly bounded over all choices of $\alpha_{n+1}$ and $\beta_{n+1}$ and all $x$.  Thus for $c_{n+1}$ sufficiently small,  
the desired relationships \eqref{zeroscaleinduct2} hold for $k \leq n$ no matter the values of $\alpha_{n+1}$ and $\beta_{n+1}$.
We choose $c_{n+1}$ so that this property is satisfied and also $c_{n+1} < c_n/2$.

Second, since $\Delta h_n$ is zero in neighborhoods of $\pm 1$, there exist arbitrarily small $\sigma_{n+1}$, $\alpha_{n+1}$ and $\beta_{n+1}$
such that
\[
\Delta (h_n * G_{\sigma_{n+1}}) \left( \pm \sqrt{\sigma_{n+1}^2 + 1} \right)
\]
is arbitrarily small in magnitude relative to
\[
J_{\alpha_{n+1}, \beta_{n+1}} * G_{\sigma_{n+1}} \left( \pm \sqrt{\sigma_{n+1}^2 + 1} \right),
\]
and therefore the sign of 
\begin{multline*}
\Delta (h_{n+1} * G_{\sigma_{n+1}}) \left( \pm \sqrt{\sigma_{n+1}^2 + 1} \right) \\
= \Delta (h_n * G_{\sigma_{n+1}} ) \left( \pm \sqrt{\sigma_{n+1}^2 + 1} \right) + c_{n+1} J_{\alpha_{n+1}, \beta_{n+1}} * G_{\sigma_{n+1}} \left( \pm \sqrt{\sigma_{n+1}^2 + 1} \right)
\end{multline*}
coincides with that of 
\[
J_{\alpha_{n+1}, \beta_{n+1}} * G_{\sigma_{n+1}} \left( \pm \sqrt{\sigma_{n+1}^2 + 1} \right).
\]

Finally, since $\alpha_{n+1}$ and $\beta_{n+1}$ were chosen to satisfy $0 < (-1)^{n+1} \beta_{n+1} < \alpha_{n+1}$, it follows that
\[
(-1)^n J_{\alpha_{n+1}, \beta_{n+1}} * G_{\sigma_{n+1}} \left( \pm \sqrt{\sigma_{n+1}^2 + 1} \right) > 0,
\]
and hence
\[
(-1)^n \, \Delta( h_{n+1} * G_{\sigma_{n+1}} )\left( \pm \sqrt{\sigma_{n+1}^2 + 1} \right) >0,
\]
as desired.  This completes the inductive construction of $h_{n+1}$ and $\sigma_{n+1}$.

Having defined the partial sums $h_n$ inductively, we now define $h$ to be their limit in the $L^1$ topology.  This limit exists because $h_n - h_{n-1}$ has $L^1$-norm $4 \alpha_n c_n(1+\beta_n)$ (see Figure \ref{Jfig}), and therefore the $L^1$-norm of $h_n$ is bounded for each $n$ by  $4 \sum_{m=1}^\infty c_m \alpha_m(1+\beta_m)$. This sum converges since the $c_n$ are bounded by a geometric sequence 
($c_{n+1}<c_n/2$), and $|\beta_n|<|\alpha_n| < 1$ for each $n$.

In this limit, the relationships \eqref{zeroscaleinduct} are preserved with $\leq$ in place of $<$:
\[
\begin{cases}
\Delta (h * G_{\sigma_k}) \left(\pm \sqrt{\sigma_k^2 + 1} \right) \geq 0 & \text{for $k$ odd}\\
\Delta (h * G_{\sigma_k}) \left(\pm \sqrt{\sigma_k^2 + 1} \right) \leq 0 & \text{for $k$ even.}
\end{cases}
\]
This implies that the edge contours of $G$ and $G+h$ cross infinitely often as $\sigma \rightarrow 0$.  Since the two edge contours of both $G$ and
$G + h$ are symmetric about the $\sigma$-axis, the intersections of edge contours on each side of the $\sigma$-axis occur at the same $\sigma$-values.
Thus the edges of $G$ and $G+h$ agree on an infinite sequence of scales tending to zero.  This proves Corollary \ref{cor:OneGaussian}(b).


\section{Distributions with Finitely Many Moments}
\label{FiniteMoment}

We have shown that any one-dimensional function with exponential decay is uniquely determined by a sequence of scaled Gaussian edges, thus giving a sufficient condition for the Marr conjecture in one dimension.  One can ask whether this result could be extended to functions that decay less rapidly---for example, functions with algebraic decay.  Addressing this question requires a formal notion of distributions with only finitely many moments.  To that end, this section introduces the space $\cM_N$ of smooth test functions of asymptotic order $|x|^N$ or less, and its dual $\cM_N'$, whose elements are distributions with moments through order $N$.  We first define these spaces, then consider derivatives and antiderivatives of distributions in $\cM_N'$, and finally we prove the existence and continuity of asymptotic moment expansions for such distributions.  

We consider only one-dimensional distributions, but the definitions and results presented here can readily be generalized to arbitrary dimensions.

\subsection{Definitions}

For any nonnegative integer $N$, let $\cM_N$ denote the space of smooth test functions $\psi$ on $\R$ such that, for each integer $n \geq 0$, the seminorm
\begin{equation}
\label{Mseminorm}
|| \psi ||_{N,n} = \sup_{x \in \R} \; (1+|x|)^{-(N-n)} |\psi^{(n)}(x)|
\end{equation}
is finite.  (These seminorms were first introduced by H\"{o}rmander \cite{Hormander} and are often used to define symbol classes of pseudodifferential operators \cite[e.g.]{Taylor}.)

The topology on $\cM_N$ is generated by the family of seminorms $|| \; ||_{N,n}$ for $n \geq 0$.  Functions in $\cM_N$ behave asymptotically as $|x|^N$ or less, and their $n$th derivatives behave asymptotically as $|x|^{N-n}$ or less.  In particular, $x^m \in \cM_N$ for each integer $0 \leq m \leq N$.  We also note from \eqref{Mseminorm} that for $M \leq N$, $|| \psi ||_{N,n} \leq || \psi ||_{M,n}$ for each $\psi \in \cM_M$ and $n \geq 0$, and it follows that $\cM_M \subset \cM_N$.

We denote the dual space of distributions on $\cM_N$ by $\cM_N'$.  Distributions in $\cM_N'$ have moments through order $N$, where the $n$th moment of $f \in \cM_N'$ is defined as
\[
\mu_n(f) \equiv \langle f,x^n \rangle.
\]
For $M \leq N$, we have $\cM_N' \subset \cM_M'$ since $\cM_M \subset \cM_N$. We also
note that for all $N$, $\cM_N \subset \cP$ and hence $\cP' \subset \cM_N'$.

\subsection{Derivatives and Antiderivatives}

We observe from \eqref{Mseminorm} that for each $n \geq 0$, $0 \leq m \leq N$,
\begin{equation}
\label{seminormder}
|| \psi^{(m)} ||_{N-m,n}=||\psi ||_{N,n+m},
\end{equation}
and therefore, $\psi^{(m)} \in \cM_{N-m}$ whenever $\psi \in \cM_N$.  This relation also shows that the derivative is a continuous linear functional from $\cM_N$ to $\cM_{N-1}$.  

The derivative of a distribution $f \in \cM_N'$ is defined as an element of $f' \in \cM_{N+1}'$ by the relation
\begin{equation}
\label{distder}
\langle f', \psi \rangle = - \langle f, \psi' \rangle,
\end{equation}
for all $\psi \in \cM_{N+1}$.  By extension, the $m$th derivative of $f \in \cM_N'$, denoted $f^{(m)}$, is an element of $\cM_{N+m}'$, for each integer $m \geq 0$.

We can also define the antiderivative of a distribution $f \in \cM_N'$, provided that $f$ has vanishing zeroth moment.  This definition requires the following lemma regarding antiderivatives of test functions:

\begin{lemma}
\label{testantider}
If $\psi$ is a smooth function and $\psi' \in \cM_{N-1}$, $N \geq 1$, then $\psi \in \cM_N$.
\end{lemma}

\begin{proof}
Since $||\psi||_{N,n}=||\psi'||_{N-1,n-1}$ for all $n \geq 1$, we need only verify that $||\psi||_{N,0}$ is finite.  To show this, we note that $\psi' \in \cM_{N-1}$ implies that $||\psi' ||_{(N-1,0)}$ is finite, and thus there exists some constant $C>0$ such that $|\psi'(x)| \leq C(1+|x|)^{N-1}$ for all $x \in \R$.  In particular, we have
\begin{subequations}
\begin{align}
\label{phi'bounda}
\psi'(x) \leq C(1+x)^{N-1} \qquad & \text{for $x>0$}\\
\label{phi'boundb}
\psi'(x) \geq -C(1+x)^{N-1} \qquad & \text{for $x>0$}\\
\label{phi'boundc}
\psi'(x) \leq C(1-x)^{N-1} \qquad & \text{for $x<0$}\\
\label{phi'boundd}
\psi'(x) \geq -C(1-x)^{N-1} \qquad & \text{for $x<0$}.
\end{align}
\end{subequations}
Upon integrating both sides of \eqref{phi'bounda} and \eqref{phi'boundb} from $x=0$ to $x=\infty$, and \eqref{phi'boundc} and \eqref{phi'boundd} from $x=-\infty$ to $x=0$ (and recalling that $N \geq 1$), it follows that there exists some $K$ such that $|\psi(x)| \leq K(1+|x|)^N$.  Thus $||\psi||_{N,0}$ is finite, completing the proof.
\end{proof}

Using the above lemma, we show that any $f \in \cM_N'$ with $\mu_0(f)=0$ has an antiderivative in $\cM'_{N-1}$.

\begin{corollary}
\label{distantider}
If $f \in \cM_N'$, $N \geq 1$, and $\mu_0(f)=0$, then there exists a unique $g \in \cM'_{N-1}$ with $g'=f$.
\end{corollary}

\begin{proof}
For $\psi \in \cM_{N-1}$, define $\langle g, \psi \rangle = - \langle f, \psi \rangle$, where $\psi$ is an antiderivative of $\psi$.  The quantity $\langle f, \psi \rangle$ is well-defined since $\psi \in \cM_N$ by Lemma \ref{testantider}.  Since
\[
\langle f, \psi +C \rangle = \langle f, \psi \rangle + C \mu_0(f) = \langle f, \psi \rangle,
\]
the value of $\langle g, \psi \rangle$ does not depend on the choice of antiderivative. 

To show that $g$ is a continuous functional on $\cM_{N-1}$, consider a sequence $\{\psi_i \in \cM_{N-1}\}_{i\geq 1}$ converging to the zero function in the topology of $\cM_{N-1}$.  We define a corresponding sequence $\{\psi_i \in \cM_N \}_{i\geq 1}$ by
\[
\psi_i(x) = \int_{-\infty}^x \psi_i(y) \; dy \qquad \text{for each $i$.}
\]

We claim that $\{\psi_i \}_{i \geq 1}$ converges as $i \to \infty$ to the zero function in the topology of $\cM_N$.  Indeed, for $n \geq 1$ we have from \eqref{seminormder} that 
\[
\lim_{i \to \infty} || \psi_i ||_{N,n} = \lim_{i \to \infty}|| \psi_i ||_{N-1,n-1} = 0.
\]
It therefore only remains to show that $\lim_{i \to \infty} || \psi_i ||_{N,0} = 0$.  This can be shown by observing that, since $\lim_{i \to \infty} || \psi_i ||_{N-1,0} = 0$, there is a sequence of positive numbers $\{C_i\}_{i \geq 1}$, $C_i \to 0$, with $\psi_i(x)$ bounded in absolute value by $C_i(1+|x|)^{N-1}$.  Integrating separately over the domains $(-\infty,0]$ and $[0,\infty)$ as in the proof of Lemma \ref{testantider}, it follows that there is a sequence of positive numbers $\{ K_i \}_{i \geq 1}$, $K_i \to 0$, such that $\psi_i(x)$ is bounded in absolute value by $K_i(1+|x|)^N$.  This proves that $\lim_{i \to \infty} || \psi_i ||_{N,0} = 0$ and thereby verifies the claim that $\{ \psi_i \}_{i \geq 1}$ converges to the zero function in the topology of $\cM_N$.

The continuity of $g$ as a functional on $\cM_{N-1}$ now follows from its definition and the continuity of $f$:
\[
\lim_{i \to \infty} \langle g, \psi_i \rangle = - \lim_{i \to \infty} \langle f, \psi_i \rangle = 0.
\]
We conclude that $g$ is unique and well-defined as an element of $\cM_{N-1}'$.
\end{proof}

Iterating Corollary \ref{distantider}, a distribution in $\cM_N'$ whose first $m$ moments vanish has a unique $m$th antiderivative in $\cM'_{N-m}$:

\begin{corollary}
\label{mthantider}
Consider $f \in \cM_N'$, $N \geq 1$ such that $\mu_0(f)=\ldots=\mu_{m-1}(f)=0$, for some positive integer $m \leq N$.  Then there exists a unique $g \in \cM'_{N-m}$ with $g^{(m)}=f$.  
\end{corollary}

\begin{proof}
We proceed by induction on $m$, with Corollary \ref{distantider} serving as a base ($m=1$) case.  Suppose the claim holds in the case $m=m^*$ for some positive integer $m^*\leq N-1$;  we will prove it for $m=m^*+1$. Consider $f \in \cM_N'$ with $\mu_0(f)=\ldots=\mu_{m^*}(f)=0$.  By the inductive hypothesis, there exists a unique $h \in \cM_{N-m^*}'$ with $h^{(m^*)}=f$.  Iteratively applying \eqref{distder} $m^*$ times, we have that
\begin{align*}
0 & = \mu_{m^*}(f)\\
& = \left \langle h^{(m^*)}, x^{m^*} \right \rangle\\
& = (-1)^{m^*} (m^*)! \; \langle h, 1\rangle\\
& = (-1)^{m^*} (m^*)! \; \mu_0(h).
\end{align*}
Thus $\mu_0(h)=0$, which allows us to apply Corollary \ref{distantider} to $h$. We obtain that there exists a unique $g \in \cM'_{N-m^*-1}$ with $g'=h$.  Taking $m^*$ derivatives of both sides yields $g^{(m^*+1)} = f$, completing the induction step.
\end{proof}

\subsection{Asymptotic moment expansion}

Here introduce the asymptotic moment expansion for distributions with finitely many moments.  A distribution $f \in \cM_N'$ has an asymptotic moment expansion to order $N-1$ in the moments of $f$, convolutions of which  converge locally uniformly, as we show in the an analogue of Theorem \ref{thm:momentcontinuity}:

\begin{finitethm}
For all integers $0 \leq M \leq N-1$ and $f \in \cM_N',\psi \in \cM_N$, the $\sigma$-indexed family of functions 
\begin{equation*}
w \longmapsto \sigma^{M+1} \left( f * \psi_\sigma (\sigma w)
- \sum_{n=0}^{M} \frac{(-1)^n}{n!} \; \mu_n \; \sigma^{-n-1} \; \psi^{(n)} (w) \right),
\end{equation*}
converges locally uniformly (in $w$) to the zero function (of $w$) as $\sigma \rightarrow \infty$.
\end{finitethm}
Once the following analogue of Lemma \ref{rholemma} is proved, the proof of Theorem \ref{finitemomentexp} follows exactly the proof of Theorem \ref{thm:momentcontinuity} (in the case $d=1$, which is the only case we consider here).

\begin{lemma}
Let $\rho(w,y)$ be a smooth function and, for fixed $w \in \R$, define $\rho_w$ by $\rho_w(y)=\rho(w,y)$.  Suppose that
\renewcommand{\labelenumi}{(\alph{enumi})}
\begin{enumerate}
\item $\rho_w \in \cM_N$ for some fixed $N \geq 1$,
\item For each $n \geq 0$, $||\rho_w||_{N,n}$ is locally uniformly bounded in $w$, and
\item There is some integer $M$, $0 \leq M \leq N-1$, such that for each $w \in \R$ and $0 \leq m \leq M$,
\[
\frac{d^m }{d y^m} \rho_w(0)= 0.
\]
\end{enumerate}
Then for any continuous seminorm $||\quad ||$ on $\cM_N$, the $\sigma$-indexed family of functions 
\[
w \longmapsto \sigma^M ||\rho_w( \placehold /\sigma)|| 
\]
converges locally uniformly (in $w$) to the zero function (of $w$) as $\sigma \rightarrow \infty$.
\end{lemma}

(We recall that the symbol $\placehold$ represents function or distribution arguments with regard to the bracket and seminorm operations.)

\begin{proof}
Suppose the conclusion is false for the seminorm $|| \quad ||_{N,n}$.  Then there is a compact neighborhood $K \subset \R$ and a pair of sequences $\{w_j \in K\}_{j \geq 0}$, $\{\sigma_j \in \R\}_{j \geq 0}$, with $\sigma_j \to \infty$, such that 
\[
0 < \lim_{j \rightarrow \infty} \sigma_j^M \left| \left|\rho_{w_j} (\placehold /\sigma_j)\right | \right|_{N,n} =
\lim_{j \rightarrow \infty} \sigma_j^{M-n} \, \sup_{y \in \R} \, (1+|y|)^{-(N-n)} \, |\rho_{w_j}^{(n)} (y/\sigma_j)|.
\]
The second equality above uses 
\[
\frac{d^n}{dy^n} \big(  \rho_{w_j}(y/\sigma_j) \big) = \sigma_j^{-n} \rho_{w_j}^{(n)} (y/\sigma_j).
\]
By passing to a subsequence if necessary, we may assume $\{w_j\}$ converges to a $w' \in K$.  Since for  fixed $w_j$, values of $y$ can be chosen to make the quantity
\[
(1+|y|)^{-(N-n)} \, |\rho_{w_j}^{(n)} (y/\sigma_j)|
\]
arbitrary close to its supremum over $y \in \R$, there is a sequence $\{y_j\}_{j \geq 0}$ such that
\begin{equation}
\label{rholambdaw}
\lim_{j \rightarrow \infty} \sigma_j^{M-n} \,  (1+|y_j|)^{-(N-n)} \, |\rho_{w_j}^{(n)} (y_j/\sigma_j)| >0.
\end{equation}

Passing to further subsequences if necessary,
we may assume that $\{y_j/\sigma_j\}$ either converges to 0 or is bounded away from 0 in absolute value as $j \to \infty$.

\paragraph*{Case 1: $\lim_{j \to \infty} y_j/\sigma_j = 0$ and $n \leq M$.}  In this case we rewrite \eqref{rholambdaw} as
\begin{equation}
\label{rholambdaw1}
\lim_{j \rightarrow \infty} \left( \frac{|y_j|^{M-n}}{(1+|y_j|)^{N-n}} \right) 
\Big( |y_j/\sigma_j|^{-(M-n)} \, |\rho_{w_j}^{(n)}(y_j/\sigma_j)| \Big) >0.
\end{equation}
Above, the first parenthesized quantity $\frac{|y_j|^{M-n}}{(1+|y_j|)^{N-n}}$ is bounded above by 1 for all $j$ since $0 \leq M-n<N-n$.  For the second parenthesized quantity, we have that
\[
\lim_{j \rightarrow \infty} |y_j/\sigma_j|^{-(M-n)} \, |\rho_{w'}^{(n)}(y_j/\sigma_j)|=0,
\]
by condition (c) of the statement of the lemma, so 
\[
\lim_{j \rightarrow \infty} |y_j/\sigma_j|^{-(M-n)} \, |\rho_{w_j}^{(n)}(y_j/\sigma_j)|=0,
\]
by the smoothness of $\rho$ in both arguments.  Thus
\[
\lim_{j \rightarrow \infty} \left( \frac{|y_j|^{M-n}}{(1+|y_j|)^{N-n}} \right) 
\Big( |y_j/\sigma_j|^{-(M-n)} \, |\rho_{w_j}^{(n)}(y_j/\sigma_j)| \Big) =0,
\]
contradicting \eqref{rholambdaw1}.

\paragraph*{Case 2: $\lim_{j \to \infty} y_j/\sigma_j = 0$ and $n > M$.}  In this case we rewrite \eqref{rholambdaw} as
\begin{equation}
\label{rholambdaw2}
\lim_{j \rightarrow \infty}  \frac{(\sigma_j^{-1} + |y_j/\sigma_j|)^{n-M}}{(1+|y_j|)^{N-M}} \;
|\rho_{w_j}^{(n)}(y_j/\sigma_j)| > 0.
\end{equation}
The quantity
\[
\frac{(\sigma_j^{-1} + |y_j/\sigma_j|)^{n-M}}{(1+|y_j|)^{N-M}}
\]
converges to 0 as $j \to \infty$ since $\sigma_j^{-1}$ and $y_j/\sigma_j$ both converge to 0, and $n-M$ and $N-M$ are both positive in this case.  On the other hand, 
\[
\lim_{j \to \infty} |\rho_{w_j}^{(n)}(y_j/\sigma_j)| = |\rho_{w'}^{(n)}(0)|,
\]
by the smoothness of $\rho$.  Thus
\[
\lim_{j \rightarrow \infty} \frac{(\sigma_j^{-1} + |y_j/\sigma_j|)^{n-M}}{(1+|y_j|)^{N-M}} |\rho_{w_j}^{(n)}(y_j/\sigma_j)| = 0,
\]
contradicting \eqref{rholambdaw2}.

\paragraph*{Case 3: $|y_j/\sigma_j| > B$ for some $B>0$ and all $j \geq 0$.}  In this case, we rewrite \eqref{rholambdaw} as 
\begin{equation}
\label{rholambdaw3}
\lim_{j \rightarrow \infty} \left( \frac{\sigma_j^{M-n}(1+|y_j/\sigma_j|)^{N-n}}{(1+|y_j|)^{N-n}} \right)
\Big( (1+|y_j/\sigma_j|)^{-(N-n)} |\rho_{w_j}^{(n)}(y_j/\sigma_j)| \Big) >0.
\end{equation}
The second parenthesized quantity in \eqref{rholambdaw3},
\[
 (1+|y_j/\sigma_j|)^{-(N-n)} |\rho_{w_j}^{(n)}(y_j/\sigma_j)|,
\]
is positive, less than or equal to $|| \rho_{w_j}||_{N,n}$ by this norm's definition, and therefore bounded in $j$ since $|| \rho_w||_{N,n}$ is locally uniformly bounded in $w$.  As for the first parenthesized quantity, since $|y_j/\sigma_j|>B$ implies $\lim_{j \to \infty} |y_j|=\infty$, 
\begin{align*}
\lim_{j \rightarrow \infty}  \frac{\sigma_j^{M-n}(1+|y_j/\sigma_j|)^{N-n}}{(1+|y_j|)^{N-n}}
& = \lim_{j \rightarrow \infty}  \frac{\sigma_j^{M-n}(1+|y_j/\sigma_j|)^{N-n}}{|y_j|^{N-n}}\\
& = \lim_{j \rightarrow \infty} \sigma_j^{M-N} ( |y_j/\sigma_j|^{-1}+1)^{N-n}\\
& = 0.
\end{align*}
The last equality follows from the facts that $|y_j/\sigma_j|^{-1} < B^{-1}$ for all $j$, and $\sigma_j^{M-N} \to 0$ since $M<N$.  We conclude 
\[
\lim_{j \rightarrow \infty} \left( \frac{\sigma_j^{M-n}(1+|y_j/\sigma_j|)^{N-n}}{(1+|y_j|)^{N-n}} \right)
\Big( (1+|y_j/\sigma_j|)^{-(N-n)} |\rho_{w_j}^{(n)}(y_j/\sigma_j)| \Big) =0,
\]
contradicting \eqref{rholambdaw3}.

We have shown that the $\sigma$-indexed family of functions 
\[
w \longmapsto \sigma^M ||\rho_w(\placehold /\sigma)||_{N,n}
\]
converges locally uniformly (in $w$) to the zero function of $w$ as $\sigma \rightarrow \infty$ for each $n \geq 0$.  Since the family of seminorms 
$|| \; ||_{N,n}$ generates the topology on $\cM_N$, the result is true for any continuous seminorm.
\end{proof}

For the Gaussian wavelet $\psi=G$ we have:
\begin{corollary}
\label{Gaussianfinitemoment}
For all $f \in \cM_N'$, $N \geq 0$, and all $M$, $0 \leq M \leq N-1$, the family of functions 
\[
w \mapsto \sigma^{M+1} 
\left( f*G_\sigma (\sigma w) - \sum_{n=0}^{M} \frac{\mu_n}{n!} \sigma^{-n-1} H_n(w) G(w) \right)
\]
converges locally uniformly to the zero function of $w$ as $\sigma \to \infty$.
\end{corollary}

\section{Necessity of Strong Decay}
\label{necessity}

Corollary \ref{cor:OneGaussian}(a) states that a real-valued function with exponential decay is uniquely determined by its Gaussian edges at any sequence of scales not converging to zero.  On the other hand, Meyer's counterexample \cite{Meyer} shows that such unique determination fails for non-decaying functions.  This raises the question of the requirements on a function $f$ for it to be uniquely determined by a sequence of its Gaussian edges.  One might conjecture that unique determination can be extended to all functions vanishing at infinity.

Here we prove Corollary \ref{cor:OneGaussian}(c), showing that the above conjecture is false.  The proof will proceed by constructing a sequence of pairs of distributions, with an arbitrarily large fixed number of moments, whose limits have Gaussian edges coinciding on an infinite sequence of scales tending to infinity.  Thus the unique determination result does not, in general, extend to functions with algebraic decay, leaving open only classes of functions with decay rates between exponential and algebraic, e.g.~classes decaying as the 
log-normal function $f(x)=\frac{1}{x} e^{-(\ln|x|)^2}$ or faster.

Let $N$ be a positive multiple of 4, and consider a positive symmetric function $h \in L^1(\R)$ satisfying the following conditions:
\renewcommand{\labelenumi}{(\roman{enumi})}
\begin{enumerate}
\item For all $\phi \in \cM_{N-1},$ 
\begin{equation}
\label{hfinitemoment}
\int_\R |\phi(x)| h(x) \; dx < \infty.
\end{equation}
(Thus $h$ can be regarded as an element of $\cM_{N-1}'$.)
\item $h$ has infinite $N$th moment: $\displaystyle \int_\R x^N h(x) \; dx = \infty.$
\item $h$ has second moment $<2$: $\displaystyle \int_\R x^2 h(x) \; dx < 2.$
\end{enumerate}
(We note that the third condition can always be arranged by multiplying $h$ by an appropriate constant.)

Starting with any such $h$ we will construct a pair of distributions $f, g \in \cM_{N-3}'$.  We will show that $f$ and $g$ have exactly two persistent edge contours each, which are symmetric in the coordinate $w=x/\sigma$.  We will further show that there is
a sequence of pairs $\{(w_i, \sigma_i)\}_{i \geq 1}$, with $w_i, \sigma_i>0$ and $\{\sigma_i\}$ increasing, such that
\begin{equation}
\label{interweave1}
\begin{split}
  \Delta \big( f * G_{\sigma_i} \big) \left(\sigma_i w_i \right) \geq 0 \geq   \Delta \big( g * G_{\sigma_i} \big) \left(\sigma_i w_i \right) & \quad 
\text{for $i$ odd}\\
  \Delta \big ( f * G_{\sigma_i} \big) \left(\sigma_i w_i \right) \leq 0 \leq   \Delta \big( g * G_{\sigma_i} \big) \left(\sigma_i w_i \right) & \quad 
\text{for $i$ even.}
\end{split}
\end{equation}
These statements together imply that edge contours of $f$ and $g$ intersect on a sequence of scales tending to infinity.
Finally, to obtain a violation of Marr's conjecture, we replace the distributions $f, g \in \cM_{N-3}$ by the integrable functions $f*G$ and $g*G$, whose edge contours are the same as those of $f$ and $g$, but shifted by one unit in $\sigma$.

The argument consists of two parts.  The first constructs of $f$ and $g$, demonstrates the existence of two persistent, symmetric edge contours, and verifies \eqref{interweave1}.  The second part shows that $f$ and $g$ have no other persistent edge contours.  Condition (iii) above will only be invoked in the second part.

\subsection{Part 1: Construction of $f$ and $g$}
\label{Part1}

We first construct $f$, $g$, and $\{(w_i, \sigma_i)\}_{i \geq 1}$ inductively, similarly to the argument of Section \ref{finiteedges}.  At each step $k$ of the induction we will construct a pair of distributions $f_{k+1}, g_{k+1} \in \cM_{N-3}'$ and pairs $(w_{2k}, \sigma_{2k})$ and $(w_{2k+1}, \sigma_{2k+1})$ such that \eqref{interweave1} holds for $1 \leq i \leq 2k+1$, with $f_{k+1}$ and $g_{k+1}$ in place of $f$ and $g$.  After the induction, we will take the limits of $f_k$ and $g_k$ (as $k \to \infty$) to obtain $f$ and $g$.

As the base step, we construct distributions $f_1,g_1 \in \cM'_{N-3}$.  For arbitrary real numbers $d_1>c_1>0$, let $C_1,D_1 \subset \R$ be the intervals $[-c_1,c_1]$ and $[-d_1,d_1]$, respectively.  We define $f_1$ and $g_1$ by specifying their second derivatives $\Delta f_1, \Delta g_1 \in \cM_{N-1}'$:
\begin{equation}
\label{deltafg1}
\begin{split}
\Delta f_1 & = \delta^{(2)} +  \chi_{C_1} h 
- \sum_{\substack{\text{$m$ even}\\ 0 \leq m \leq N-2}} a_{1,m} \delta^{(m)}\\
\Delta g_1 & = \delta^{(2)} +  \chi_{D_1} h  
- \sum_{\substack{\text{$m$ even}\\ 0 \leq m \leq N-2}}  b_{1,m} \delta^{(m)}.
\end{split}
\end{equation}
Here, $\chi_U$ denotes the characteristic function $U \subset \R$, with value 1 on $U$ and zero elsewhere. The coefficients $a_{1,m}$ and $b_{1,m}$ in \eqref{deltafg1}, for $m$ even and $0 \leq m \leq N-2$, are set as
\begin{align*}
a_{1,m} & =   \int_{C_1} \frac{x^m}{m!} h(x)\; dx\\
b_{1,m} & =  \int_{D_1} \frac{x^m}{m!} h(x) \; dx.
\end{align*}
This guarantees that $\mu_2 \big(\Delta f_1\big)=\mu_2\big(\Delta g_1\big)=2$, and 
$\mu_n\big(\Delta f_1\big)=\mu_n\big(\Delta g_1\big)=0$ for $0 \leq n \leq N-1, n \neq 2$.  (Thus the moments of $\Delta f_1$ and $\Delta g_1$ coincide with those of $\delta^{(2)}$ to order $N-1$.  Note that the odd moments of $\Delta f_1$ and $\Delta g_1$ vanish due to the symmetry of $h$.)  In particular, since the zeroth and first moments of $\Delta f_1$ and $\Delta g_1$ are both zero, Corollary \ref{mthantider} guarantees that $f_1$ and $g_1$ are well-defined from \eqref{deltafg1} as elements of $\cM'_{N-3}$.  More strongly, since $f_1$, $g_1$, $\Delta f_1$ and $\Delta g_1$ are all compactly supported, these distributions are all elements of $\cP'$.

Expanding $\Delta f_1 * G$ and $\Delta g_1 * G$ as in \eqref{prezceqn} (with $\Delta f$ replacing $f$) and invoking \eqref{Hermite}, we can describe the edges of $f_1$ and $g_1$ in $w$ and $\sigma$ as the respective zeros of
\begin{equation}
\label{fgexpand}
\begin{split}
\Delta \big( f_1 *G_\sigma \big ) (\sigma w)  &=  \sigma^{-3} H_2(w) G(w)+ \frac{\mu_N \big(\Delta f_1 \big)}{N!} \sigma^{-N-1}H_N(w) G(w)  + \mathcal{O}(\sigma^{-N-2})\\
\Delta  \big( g_1 *G_\sigma \big ) (\sigma w)  &= \sigma^{-3} H_2(w) G(w)+ \frac{\mu_N \big(\Delta g_1 \big)}{N!} \sigma^{-N-1} H_N(w) G(w) + \mathcal{O}(\sigma^{-N-2})\\
&  \qquad (\sigma \rightarrow \infty).
\end{split}
\end{equation}

For $|w|$ close to 1, both coefficients of $\sigma^{-N-1}$ in \eqref{fgexpand} are positive.  This follows since $h$ is positive---hence so are $\mu_N(\Delta f_1)$ and $\mu_N(\Delta g_1)$---and $H_N(\pm 1)$ is positive according to \eqref{explicitHermite} for $N$ a multiple of 4 (as required).  Furthermore, since $C_1 \subset D_1$, we have $\mu_N(\Delta f_1) < \mu_N(\Delta g_1)$ and so the larger of the two coefficients of $\sigma^{-N-1}$ is that associated to $g_1$.  We conclude from this analysis of the coefficients in \eqref{fgexpand} that for any fixed $w \approx \pm 1$, 
\begin{equation}
\label{basecondition0}
\Delta \big( f_1 * G_{\sigma} \big ) \left(\sigma w \right) < \Delta \big( g_1 * G_{\sigma} \big) \left(\sigma w \right)
\end{equation}
for all sufficiently large $\sigma$.

We also observe from \eqref{fgexpand} that $f_1$ and $g_1$ each have (at least) two persistent edge 
contours,  corresponding to the roots $w = \pm 1$ of $H_2(w)=w^2-1$.   By Corollary \ref{onetoone}, there is a unique value of $w$ corresponding to each $\sigma>0$ for each of these edge contours. In particular, by the symmetry of $f_1$ and $g_1$, these edge contours can be parameterized as $w  = \pm e_{f_1}(\sigma)$ and $w = \pm e_{g_1}(\sigma)$.  Since the coefficients of $\sigma^{-N-1}$  in \eqref{fgexpand} are both positive for $|w| \approx 1$ as previously stated, and the coefficients of $\sigma^{-3}$ have the sign of $H_2(w)=w^2-1$, $e_{f_1}(\sigma)$ and $e_{g_1}(\sigma)$ both approach 1 from below as $\sigma \to \infty$ (see Figure \ref{basecaseedges}).   Therefore, for any $w_1$ less than but sufficiently close to $1$, the line $w=w_1$ intersects both edge contours described by $w = e_{f_1}(\sigma)$ and $w = e_{g_1}(\sigma)$.  Combining this observation with \eqref{basecondition0} implies that for $w_1$ less than but sufficiently close to $1$, there is a range of $\sigma$ values satisfying
\begin{equation}
\label{basecondition}
\Delta \big( f_1 * G_{\sigma} \big) \left(\sigma w_1 \right) < 0 <  \Delta \big( g_1 * G_{\sigma} \big) \left(\sigma w_1 \right).
\end{equation}
(See Figure \ref{basecaseedges}.)  Moreover, the upper bound of $\sigma$ values satisfying \eqref{basecondition} increases without bound as $w_1$ increases to 1.  
Fix $w_1$ and $\sigma=\sigma_1$ such that \eqref{basecondition} is satisfied.  We have thus constructed $f_1,g_1 \in \cP' \subset \cM_{N-3}'$ and the pair $(w_1, \sigma_1)$, which collectively serve as a base step for our iterative construction of $f,g \in \cM_{N-3}'$.

\begin{figure}
\begin{center}
\includegraphics{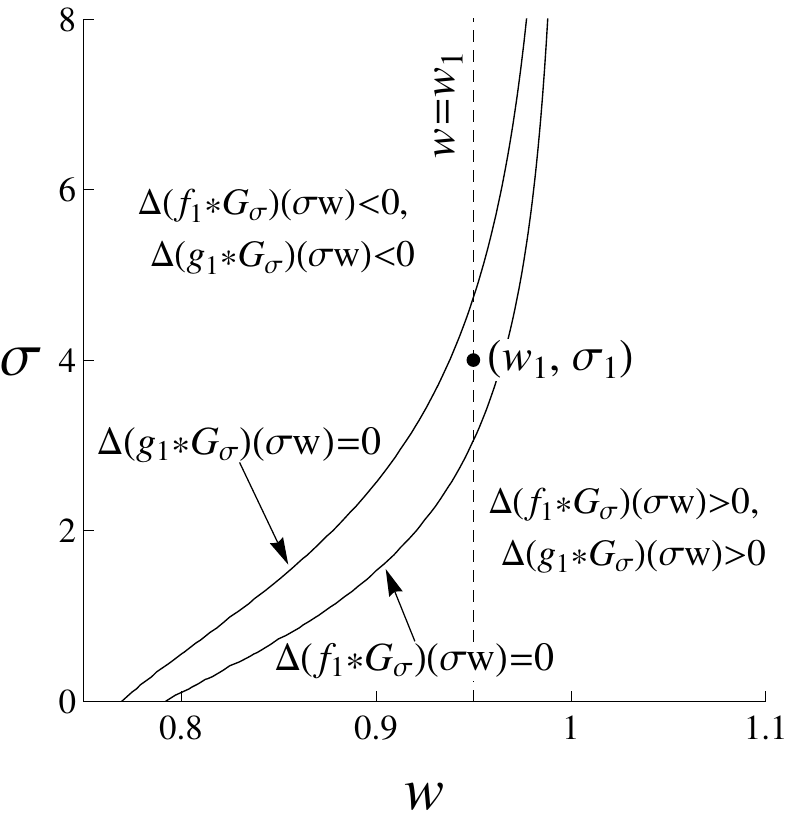}
\caption{The edge contours of example distributions $f_1$ and $g_1$ are shown together with a choice of $w_1$ and $\sigma_1$.}
\label{basecaseedges}
\end{center}
\end{figure}

As the first half of the induction step, we will give for $k \geq 1$ a construction of $f_{k+1}$ from $f_k$ and $g_k$.  The pair $(w_{2k}, \sigma_{2k})$ will be constructed along with $f_{k+1}$.  Once this is accomplished we will contruct, for the second half of the induction, $g_{k+1}$ and the pair $(w_{2k+1},\sigma_{2k+1})$, from $f_{k+1}$ and $g_k$.

First, as an inductive hypothesis, we suppose that for some $k \geq 1$, $f_k, g_k \in \cP' \subset \cM_{N-3}'$ are distributions defined by
\begin{align*}
\Delta f_k & = \delta^{(2)} +  \chi_{C_k} h 
- \sum_{\substack{\text{$m$ even}\\ 0 \leq m \leq N-2}} a_{k,m} \delta^{(m)} \\
\Delta g_k & = \delta^{(2)} +  \chi_{D_k} h 
- \sum_{\substack{\text{$m$ even}\\ 0 \leq m \leq N-2}} b_{k,m} \delta^{(m)}, 
\end{align*}
where $C_k , D_k \subset \R$ are compact and symmetric about the origin, and
\begin{align*}
a_{k,m} & =  \int_{C_k} \frac{x^m}{m!} h(x) \; dx\\
b_{k,m} & =  \int_{D_k} \frac{x^m}{m!} h(x) \; dx, \qquad 0 \leq m \leq N-2,
\end{align*}
so that as in the base case, the moments of $\Delta f_k$ and $\Delta g_k$ agree with those of $\delta^{(2)}$ to order $N-1$.
Suppose as a further inductive hypothesis that there are pairs $(w_1, \sigma_1), \ldots, (w_{2k-1}, \sigma_{2k-1})$, with $\sigma_i$ increasing in $i$, satisfying
\begin{align}
\label{interweave2}
  \Delta \big( f_k * G_{\sigma_i} \big) (\sigma_i w_i ) < 0 <    \Delta \big( g_k * G_{\sigma_i} \big) (\sigma_i w_i ) & \quad 
\text{for $i$ odd}\\
\nonumber
  \Delta \big( f_k * G_{\sigma_i} \big) (\sigma_i w_i ) > 0 >    \Delta \big( g_k * G_{\sigma_i} \big) (\sigma_i w_i ) & \quad 
\text{for $i$ even,} \quad 1 \leq i \leq 2k-1.
\end{align}

The first goal of this induction step is to construct a distribution $f_{k+1} \in \cP'$ of the same form as above,
\[
\Delta f_{k+1} = \delta^{(2)} +  \chi_{C_{k+1}} h 
- \sum_{\substack{\text{$m$ even}\\ 0 \leq m \leq N-2}} a_{k+1,m} \delta^{(m)},
\]
with
\begin{equation}
\label{akp1}
a_{k+1,m} =  \int_{C_{k+1}} \frac{x^m}{m!} h(x) \; dx,
\end{equation}
such that
\renewcommand{\labelenumi}{(\alph{enumi})}
\begin{enumerate}
\item the relationships \eqref{interweave2} are preserved,
\begin{align}
\label{interweave3}
  \Delta \big( f_{k+1} * G_{\sigma_i} \big) (\sigma_i w_i ) < 0 <   \Delta \big( g_k * G_{\sigma_i} \big)  (\sigma_i w_i ) & \quad 
\text{for $i$ odd}\\
\nonumber
  \Delta \big(  f_{k+1} * G_{\sigma_i} \big) (\sigma_i w_i ) > 0 >   \Delta \big(  g_k * G_{\sigma_i} \big) (\sigma_i w_i ) & \quad 
\text{for $i$ even,} \quad 1 \leq i \leq 2k-1,
\end{align}
\item $\mu_N ( \Delta f_{k+1}) > \mu_N (\Delta g_k) + 1$.
\end{enumerate}

We do this by setting $C_{k+1} = C_k \cup C_{k+1}'$ where $C_{k+1}' = [-c_{k+1}, -c_{k+1}'] \cup [c_{k+1}', c_{k+1}]$ for some appropriately chosen positive real numbers $c_{k+1}$ and $c_{k+1}'$ with $c_{k+1} > c_{k+1}' > c_k$,
to be determined later. Note 
\begin{align}
\nonumber
& \quad \big(  ( \Delta f_{k+1} - \Delta f_k )  * G_\sigma \big) (\sigma w)\\
\nonumber
& = \left(  \left(  \chi_{C_{k+1}'} h - \sum_{\substack{\text{$m$ even}\\0 \leq m \leq N-2}} a_{k+1,m}' \delta^{(m)}\right) * G_\sigma \right) (\sigma w)\\
\label{makesmall}
& =  \int_{C_{k+1}'} h(y) G_\sigma (\sigma w - y) \; dy - \sum_{\substack{\text{$m$ even}\\0 \leq m \leq N-2}} a_{k+1,m}' G_\sigma^{(m)} (\sigma w) ,
\end{align}
where
\[
a_{k+1,m}' = \int_{C_{k+1}'} \frac{x^m}{m!} h(x) \; dx = 2 \int_{c_{k+1}'}^{c_{k+1}} \frac{x^m}{m!} h(x) \; dx.
\]

By \eqref{hfinitemoment}, the integrals $\int_\R \frac{x^m}{m!} h(x) \; dx$ converge for all nonegative integers $m \leq N-1$.  It follows that the coefficients $a_{k+1,m}'$ can be made arbitrarily small uniformly over all choices of $c_{k+1}$, by choosing a sufficiently large value of $c_{k+1}'$.  The decay properties of $G_\sigma$ and integrability of $h$ imply that for any fixed $\sigma$ and $w$, the first term of \eqref{makesmall}---and hence the full quantity \eqref{makesmall}---can also be made arbitrarily small uniformly over $c_{k+1}$, by a sufficiently large choice of $c_{k+1}'$.  Since this holds in particular for $w=w_i$ and $\sigma=\sigma_i$, $1 \leq i \leq 2k-1$, we can choose $c_{k+1}'$ such that condition \eqref{interweave3} holds regardless of the value later chosen for $c_{k+1}$, validating condition (a).  We fix such a $c_{k+1}'$.  Then since $h$ is positive and has divergent $N$th moment, a sufficiently large choice of $c_{k+1}$ will guarantee $\mu_N ( \chi_{C_{k+1}}h) > \mu_N (\chi_{D_k} h) + 1$, and hence $\mu_N ( \Delta f_{k+1}) > \mu_N (\Delta g_k) + 1$, validating condition (b).

We now construct the pair $(w_{2k}, \sigma_{2k})$.  By our choice of the coefficients $a_{k+1,m}$ in \eqref{akp1}, the moments of $\Delta f_{k+1}$ coincide with those of $\delta^{(2)}$ through order $N-1$ (as do the moments of $\Delta g_k$ according to our inductive assumption).  Furthermore, condition (b) implies that $\mu_N(\Delta f_{k+1}) > \mu_N(\Delta g_k)$.  These observations enable us, using an argument similar to that used in the base case above, to choose $w_{2k}>0$ and $\sigma_{2k}>\sigma_{2k-1}+1$ satisfying
\begin{equation*}
\label{fextend}
\Delta (f_{k+1} * G_{\sigma_{2k}}) \left(\sigma_{2k} w_{2k} \right) > 0 
> \Delta (g_k * G_{\sigma_{2k}}) \left(\sigma_{2k} w_{2k} \right).
\end{equation*}

This finishes the first half of the induction. We observe that since $\Delta f_{k+1}$ is compactly supported, it is in $\cP'$ and hence also in $\cM_{N-1}'$.

For the second half we construct, in similar fashion, a distribution $g_{k+1} \in \cP' \subset \cM_{N-3}'$ satisfying
\[
\Delta g_{k+1} = \delta^{(2)} +  \chi_{D_{k+1}} h 
- \sum_{\substack{\text{$m$ even}\\ 0 \leq m \leq N-2}} b_{k+1,m} \delta^{(m)},
\]
with
\[
b_{k+1,m} =  \int_{D_{k+1}} \frac{x^m}{m!} h(x) \; dx,
\]
where $D_{k+1} \subset \R$ is compact and symmetric about the origin, such that
\renewcommand{\labelenumi}{(\alph{enumi})}
\begin{enumerate}
\item the relationships \eqref{interweave2} are preserved now for $i$ up to $2k$ rather than $2k-1$,
\begin{align*}
  \Delta \big( f_{k+1} * G_{\sigma_i} \big) (\sigma_i w_i ) < 0 <   \Delta \big( g_{k+1} * G_{\sigma_i} \big)  (\sigma_i w_i ) & \quad 
\text{for $i$ odd}\\
  \Delta \big(  f_{k+1} * G_{\sigma_i} \big) (\sigma_i w_i ) > 0 >   \Delta \big(  g_{k+1} * G_{\sigma_i} \big) (\sigma_i w_i ) & \quad 
\text{for $i$ even,} \quad 1 \leq i \leq 2k,
\end{align*}
\item $\mu_N (\Delta g_{k+1}) > \mu_N (\Delta f_{k+1}) + 1$.
\end{enumerate}
After fixing $g_{k+1}$ we choose $w_{2k+1}$ and $\sigma_{2k+1}>\sigma_{2k}+1$ such that 
\begin{equation*}
\label{gextend}
\Delta (f_{k+1} * G_{\sigma_{2k+1}}) \left(\sigma_{2k+1} w_{2k+1} \right) < 0 
< \Delta (g_{k+1} * G_{\sigma_{2k+1}}) \left(\sigma_{2k+1} w_{2k+1} \right).
\end{equation*}
To summarize, we have constructed distributions $f_{k+1}, g_{k+1} \in \cP' \subset \cM_{N-3}'$ and pairs $(w_{2k}, \sigma_{2k})$ and $(w_{2k+1},\sigma_{2k+1})$ such that \eqref{interweave2} holds with $k$ replaced by $k+1$.  This completes the induction step.  

With the above induction argument we have constructed sequences of distributions $\{f_k\}_{k \geq 1}$, $\{g_k\}_{k \geq 1}$ and pairs $\{(w_i, \sigma_i)\}_{i \geq 1}$ such that \eqref{interweave2} holds for all values of $k$.  We claim that the sequence $\big\{\Delta f_k \big\}_{k \geq 1}$ converges in the weak-* topology on $\cM_{N-1}'$ to the distribution
\begin{equation}
\label{fdef}
\Delta f  = \delta^{(2)} +  \chi_C h - \sum_{\substack{\text{$m$ even}\\ 0 \leq m \leq N-2}} a_m \delta^{(m)},
\end{equation}
where
\begin{equation*}
\label{am}
a_m =  \int_C \frac{x^m}{m!} h(x) \; dx, \qquad C=\bigcup_{k=1}^\infty C_k,
\end{equation*}
and similarly for $\big\{\Delta g_k\big\}_{k \geq 1}$, with $D = \bigcup_k D_k$ in place of $C$.  To verify this claim, consider an arbitrary test function $\phi \in \cM_{N-1}$.  For  each $k \geq 0$ we have
\begin{equation}
\label{fkphi}
\langle \Delta f_k, \phi \rangle 
= \phi^{(2)}(0) +  \int_\R \chi_{C_k}(x) h(x) \phi(x) \; dx 
- \sum_{\substack{\text{$m$ even}\\0 \leq m \leq N-2}} a_{k,m} \phi^{(m)}(0).
\end{equation}
The integrand $\chi_{C_k}(x) h(x) \phi(x)$ of the middle term of \eqref{fkphi} is bounded in absolute value by the function $h(x) |\phi(x)|$---which is integrable by \eqref{hfinitemoment}---and converges pointwise to $\chi_C(x) h(x) \phi(x)$.  It follows from the dominated convergence theorem that the middle term of \eqref{fkphi} converges to the finite quantity
\[
 \int_\R \chi_C(x) h(x) \phi(x) \; dx,
\]
as desired.  To verify convergence of the third term of \eqref{fkphi}, it suffices to show that for each even $m$, $0 \leq m <N$, the sequence $\{a_{k,m}\}_{k \geq 0}$ converges to $a_m$ as given by \eqref{am}.  Since each $a_{k,m}$ is a constant multiple of the integral of $\chi_{C_k}(x) \, x^m h(x)$ and $x^m \in \cM_{N-1}$ for $0 \leq m<N$, convergence of each sequence $\{a_{k,m}\}_{k \geq 0}$ follows with the same argument used to prove convergence of the middle term of \eqref{fkphi}.  We conclude that $\big\{\Delta f_k\big\}$ converges as claimed, and a similar argument establishes the convergence of $\big\{\Delta g_k\big\}$.  

We have thus constructed $\Delta f$ and $\Delta g$ as elements of $\cM_{N-1}'$.  Since $G_\sigma \in \cM_{N-1}$ for each $\sigma > 0$,  the relationships \eqref{interweave2} are preserved under the weak-* limits $\Delta f_k \to \Delta f$, $\Delta g_k \to \Delta g$, with $\leq$ in place of $<$ as in \eqref{interweave1}.  We  define $f, g \in \cM_{N-3}'$ as the second antiderivatives of $\Delta f$ and $\Delta g$ respectively. (This construction is allowed by Corollary \ref{mthantider} since the zeroth and first moments of $\Delta f$ and $\Delta g$ are zero.  It can also be shown that $f$ and $g$ are the respective limits of the sequences $\{f_k\}$ and $\{g_k\}$ in the weak-* topology on $\cM_{N-3}'$, but we will not use this fact.)  

We know the following about $f, g \in \cM_{N-3}'$: They are symmetric about the origin since $C$ and $D$ are.    It can be seen from \eqref{fdef} that the moments of $\Delta f$ coincide with those of $\delta^{(2)}$ through order $N-1$, and thus $f$ and $g$ each have a pair of persistent edge contours, also symmetric about the origin, approaching $w=\pm 1$.  Finally, since \eqref{interweave1} is satisfied, these edge contours of intersect on an infinite sequence of scales.  This completes the first part of the argument.  

We will need that, since Condition (b) on $f_{k+1}$ and $g_{k+1}$ holds for each $k$, 
\[
\mu_N (\Delta f_{k+1}) > \mu_N (\Delta g_k) + 1 \quad \text{and} \quad 
\mu_N (\Delta g_{k+1}) > \mu_N (\Delta f_{k+1}) + 1,
\]
it follows that
\begin{equation}
\label{hinfmoment}
\int_C x^N h(x) \; dx = \infty \qquad \text{and} \qquad \int_D x^N h(x) \; dx= \infty.
\end{equation}
Additionally, since we required $\sigma_{i+1}>\sigma_i+1$ for all $i \geq 1$, the sequence $\{ \sigma_i \}_{i\geq 1}$ is increasing and diverges to positive infinity.

\subsection{Part 2: Non-existence of divergent edge contours}
\label{Part2}

For the second (final) part of the argument, we must show that the persistent edge contours approaching $w=\pm 1$ are the only persistent edge contours of $f$ and $g$.  We prove this for $f$, and the statement for $g$  follows similarly.

We begin by applying the moment expansion (Corollary \ref{Gaussianfinitemoment} with $M=3$) to the distribution $\Delta f\in \cM_{N-1}'$. (Recall $N \geq 4$ and thus $M \leq N-1$ for $M=3$.)  Since the moments of $\Delta f$ coincide with those of $\delta^{(2)}$ through order $N-1$, the quantity
\[
\sigma^{3} \Delta \big( f * G_\sigma \big) (\sigma w) - H_2(w)G(w)
\]
converges to zero locally uniformly in $w$, as $\sigma \to \infty$.  Thus any persistent edge contours of $f$ must either approach the roots $w=\pm 1$ of $H_2(w)$, or diverge in $w$ as $\sigma \rightarrow \infty$.  We now show that the second case cannot occur.

Assume, to the contrary, that a persistent edge contour $Z \subset H_+$ of $f$ diverges to (without loss of generality) $+\infty$ in $w$ as $\sigma \rightarrow \infty$. Define a mapping $\sigma=s(x)$ so
that for each $x$ greater than or equal to some $x_0>0$, $(x,s(x)) \in Z$.  (There is some freedom in this construction, since a line $x=x'$ may intersect $Z$ multiple times.)  By Corollary \ref{nomaximum}, local parameterizations of $Z$ have no local maxima, so $s(x)$ can be chosen to be monotone increasing in $x$.  However, $s(x)$ is not necessarily continuous---it may jump between branches of the set-valued function $S(x)=\{ \sigma: (x,\sigma) \in Z\}$.

For all $x \geq x_0$, $(x,s(x))$ lies on an edge contour of $f$, so convolving \eqref{fdef} with $G_\sigma(x)$ and applying \eqref{Hermite} yields
\begin{multline}
\label{fedge}
0 = (\Delta f)*G_{s(x)} (x)\\
 = \big((\chi_C h)*G_{s(x)} \big) (x) + s(x)^{-2} H_2 \left(\frac{x}{s(x)} \right) G_{s(x)} (x)\\  
 - \sum_{\substack{\text{$m$ even}\\ 0 \leq m \leq N-2}} a_m s(x)^{-m} H_m \left(\frac{x}{s(x)} \right) G_{s(x)} (x).
\end{multline}
(Here expressions $ f *G_{s(x)} (x)$ are calculated by first evaluating the convolution $f*G_\sigma (x)$ and then substituting $\sigma = s(x)$. The argument of $s$ is thus not considered part of the argument of $G_{s(x)} (\placehold)$ in the convolution.)

Since $Z$ diverges to infinity in $w=x/s(x)$, we have
\[
\lim_{x \rightarrow \infty} \frac{x}{s(x)} = \infty.
\]
We consider two cases, depending on the asymptotic behavior of $x/s(x)^2$.

\paragraph*{Case 1: $\displaystyle \liminf_{x \rightarrow \infty} x/s(x)^2=0$.}  In this case we rewrite  the right-hand side of \eqref{fedge} as a sum of two expressions (separately enclosed in parentheses):
\begin{multline}
\label{twoexpressions}
\Big( \big((\chi_C h)*G_{s(x)} \big) (x) - a_0 G_{s(x)} (x) \Big) \\
+ \left( (1-a_2) s(x)^{-2} H_2 \left(\frac{x}{s(x)} \right)  - 
\sum_{\substack{\text{$m$ even}\\4 \leq m \leq N-2}}
 a_m s(x)^{-m} H_m \left(\frac{x}{s(x)} \right) \right) G_{s(x)} (x).
\end{multline}
We will show that there is an $x$ for which both of these expressions are positive, contradicting \eqref{fedge}.

For the first expression in \eqref{twoexpressions} we consider the function
\[
Q(x,\sigma) = \chi_C h * G_{\sigma} (x) - a_0 G_{\sigma} (x).
\]
We prove in Lemma \ref{Qlemma} below that for each $\sigma>0$, $Q(x,\sigma)$ has exactly two zeros in $x$, is negative for $x$ between these zeros, and is positive for $x$ outside of them.  Furthermore, the zeroth moment $\mu_0 \big( \chi_C h - a_0 \delta^{(0)} \big )$ vanishes by the definition of $a_0$, while the first moment $\mu_1 \big( \chi_C h - a_0 \delta^{(0)} \big )$ vanishes since $\chi_C h$ is symmetric.  Moment expansion (Corollary \ref{Gaussianfinitemoment} with $M=3$), applied to the distribution $\chi_C h - a_0 \delta^{(0)}$, therefore implies that the quantity
\[
\sigma^{3} \left( (\chi_C h - a_0 \delta^{(0)} ) * G_\sigma \right) (\sigma w) - \frac{\mu_2 \left(\chi_C h - a_0 \delta^{(0)} \right)}{2!} H_2(w)G(w) 
\]
converges to zero locally uniformly in $w$, as $\sigma \to \infty$.  It follows that, as $\sigma \to \infty$, the two zero curves of $Q(x,\sigma)$ approach the lines $x= \pm \sigma$, corresponding to the zeros $w=\pm 1$ of $H_2(w)$.  

Since $\lim_{x \rightarrow \infty} x/s(x) = \infty$, the point $(x, s(x))$ lies outside of the two zero curves of $Q(x,\sigma)$ for all sufficiently large $x$. Recalling that $Q$ is positive outside these curves, we have that $Q \big(x,s(x) \big)$---which is equal to the first expression of \eqref{twoexpressions}---is positive for sufficiently large $x$.

The sign of the second expression of \eqref{twoexpressions} is that of the polynomial
\begin{equation}
\label{secondexp}
(1-a_2) s(x)^{-2} H_2 \left(\frac{x}{s(x)} \right)  - 
\sum_{\substack{\text{$m$ even}\\4 \leq m \leq N-2}} a_m s(x)^{-m} H_m \left(\frac{x}{s(x)} \right).
\end{equation}
Since $\lim_{x \rightarrow \infty} x/s(x) = \infty$, each of the Hermite polynomials $H_m \big(x/s(x) \big)$ in \eqref{secondexp} becomes dominated as $x \to \infty$ by its highest-order term, $\big (x/s(x) \big)^m$.  Thus, for sufficiently large $x$, the sign of \eqref{secondexp} coincides with the sign of 
\begin{align}
\nonumber
& (1-a_2) s(x)^{-2} \left(\frac{x}{s(x)} \right)^2  - 
\sum_{\substack{\text{$m$ even}\\4 \leq m \leq N-2}} a_m s(x)^{-m} \left(\frac{x}{s(x)} \right)^m\\
\label{secondexp2}
= &  (1-a_2) \left(\frac{x}{\big(s(x) \big)^2} \right)^2  - 
\sum_{\substack{\text{$m$ even}\\4 \leq m \leq N-2}} a_m \left(\frac{x}{\big(s(x) \big)^2} \right)^m.
\end{align}
This expression is a polynomial in the variable $x/\big( s(x) \big)^2$.  Since we have assumed (for Case 1) that $\liminf_{x \rightarrow \infty} x/\big( s(x) \big)^2 = 0$, there exist arbitrarily large $x$ for which the sign of \eqref{secondexp2} coincides with the sign of its lowest-order term's coefficient $1-a_2$.  By Condition (iii) on $h$ (the bound on the second moment of $h$; see beginning of this section),
\[
1-a_2 = 1-\int_C \frac{x^2}{2!} h(x) \; dx > 0.
\]
Thus there exist arbitrarily large $x$ for which \eqref{secondexp2}---and hence also the second expression of \eqref{twoexpressions}---is positive.  Since the first expression of \eqref{twoexpressions} is positive for sufficiently large $x$, there are values of $x$ for which both expressions in \eqref{twoexpressions} are positive, contradicting \eqref{fedge}.

\paragraph*{Case 2: $\displaystyle \liminf_{x \rightarrow \infty} x/s(x)^2>0$.}  In this case we multiply both sides of \eqref{fedge} by $\sqrt{2\pi}s(x)$ and rewrite as
\begin{multline}
\label{xsrewrite2}
\left( s(x)^{-2} H_2 \left(\frac{x}{s(x)} \right)  - 
\sum_{\substack{\text{$m$ even}\\0 \leq m < N}}  a_m s(x)^{-m} H_m \left(\frac{x}{s(x)} \right) \right)
 \exp \left( -\frac{x}{2} \frac{x}{s(x)^2} \right)\\
+ \int_{C} \exp \left(-\frac{(x-y)^2}{2 s(x) ^2} \right) h(y) \; dy = 0.
\end{multline}
Since (in Case 2) $x/s(x)^2$ is bounded below for sufficiently large $x$, the quantity $\displaystyle  \exp \left( -\frac{x}{2} \frac{x}{s(x)^2} \right)$ is bounded above by an exponentially decreasing function of $x$.  Further, since $s(x)$ is monotone increasing, $s(x)^{-1}$ is bounded, and so the polynomial
\[
s(x)^{-2} H_2 \left(\frac{x}{s(x)} \right)  - 
\sum_{\substack{\text{$m$ even}\\0 \leq m \leq N-2}}  a_m s(x)^{-m} H_m \left(\frac{x}{s(x)} \right)
\]
has at most polynomial growth in $x$.  Combining these bounds, it follows that the first term of \eqref{xsrewrite2} is absolutely bounded above for all sufficiently large $x$ by a function $K e^{-\gamma x}$, with $K,\gamma>0$.
The two terms of \eqref{xsrewrite2} sum to zero, so the second term also satisfies this bound, giving
\begin{equation}
\label{secondtermbound}
Ke^{-\gamma x} > \int_{C} \exp \left( -\frac{(x-y)^2}{2 s(x) ^2} \right) h(y) \; dy,
\end{equation}
for sufficiently large $x$.  Also for $x$ sufficiently large,
\[
\exp \left( -\frac{(x-y)^2}{2 s(x) ^2} \right) > \frac{1}{2} \chi_{[-1,1]}(x-y),
\]
with $\chi_{[-1,1]}$ an indicator function as above. Combining with \eqref{secondtermbound} yields
\[
Ke^{-\gamma x} > \frac{1}{2} \int_{x-1}^{x+1} \chi_C(y) h(y) \; dy,
\]
again for sufficiently large $x$. Multiplying by $x^{N}$ and integrating from a sufficiently large $x_0$ to infinity, 
\[
K \int_{x_0}^\infty x^{N} e^{-\gamma x} \; dx > \frac{1}{2} \int_{x_0}^\infty
x^N  \int_{x-1}^{x+1} \chi_C(y) h(y) \; dy \, dx.
\]
The left-hand side is finite, thus the right-hand side is finite as well.  Interchanging order of integration on the right-hand side and noting that the integrand is nonnegative,
\begin{align*}
\frac{1}{2} \int_{x_0}^\infty x^N  \int_{x-1}^{x+1} \chi_C(y) h(y) \; dy \, dx
& \geq \frac{1}{2} \int_{x_0+1}^\infty \left(\int_{y-1}^{y+1} x^N \; dx \right) \chi_C(y) h(y) \; dy\\
& \geq  \frac{1}{2} \int_{x_0+1}^\infty 2 (y-1)^N \chi_C(y) h(y) \; dy\\
& = \int_{x_0+1}^\infty \left( \frac{y-1}{y} \right)^N y^N \chi_C(y) h(y) \; dy\\
& \geq \left(\frac{x_0}{x_0+1} \right)^N \int_{x_0+1}^\infty y^N \chi_C(y) h(y) \; dy.
\end{align*}
Thus
\[
\int_{x_0+1}^\infty y^N \chi_C (y) h(y) \; dy < \infty.
\]
Since $\chi_Ch$ is symmetric, it follows that
\[
\int_C y^N h(y) \; dy < \infty.
\]
But this contradicts the requirement \eqref{hinfmoment} that the $N$th moment of $\chi_C h$ diverges.  Thus this case is also impossible, so there are no edge contours of $f$ that diverge in $w$. 

A similar argument (with $g$ in place of $f$ and $D$ in place of $C$, starting from the beginning of Section \ref{Part2}) shows also that no edge contours of $g$ diverge in $w$.  We conclude that the only persistent edge contours of $f$ and $g$ are those that approach $w=\pm 1$.  We showed in the first part (Section \ref{Part1}) that these edge contours intersect on a sequence of scales tending to infinity.  Though $f$ and $g$ are distributions (rather than functions) we can take the convolutions $f*G$ and $g*G$ as initial functions to obtain a violation of Marr's conjecture.  This completes the proof of Corollary \ref{cor:OneGaussian}(c).  

In Section \ref{Part2}, Case 1, we made use of the following lemma (with $\tilde{h} = \chi_C h$):
\begin{lemma}
\label{Qlemma}
Let $\tilde{h} \in L^1(\R)$ be nonnegative and symmetric about $0$. Define $a_0 = \int_\R \tilde{h}(x) \; dx$, and suppose $a_0>0$.  For $x, \sigma \in \R$, $x>0$, define
\[
Q(x,\sigma) = -a_0 G_\sigma(x) + \tilde{h} * G_\sigma(x).
\]
Then for each $\sigma>0$, $Q(x,\sigma)$ has exactly two zeros in $x$, is negative between these zeros, and positive outside of them.
\end{lemma}

\begin{proof}
First we show that $Q(0,\sigma)<0$ for each $\sigma>0$.  To see this we expand
\[
Q(0,\sigma) = \frac{1}{\sqrt{2 \pi}\sigma} \left( -a_0 + \int_\R \tilde{h}(y) e^{-y^2/(2 \sigma^2)} \; dy \right).
\]
Since $a_0 = \int_\R \tilde{h}(x) \; dx$ and $e^{-y^2/(2 \sigma^2)} <1$ for each $y \neq 0$ and $\sigma > 0$, it follows that $Q(0,\sigma)<0$.

We now consider the absolute ratio of the two terms in $Q(x,\sigma)$.  This ratio can be written
\[
\left(\tilde{h} * G_\sigma(x) \right) \Big / \big(a_0 G_\sigma(x) \big)
 = \frac{1}{a_0} \int_\R \tilde{h}(y) e^{(2xy-y^2)/(2\sigma^2)} \; dy.
\]
Using symmetry of $h$, we have
\begin{equation}
\label{Qratio}
\begin{split}
\left(\tilde{h} * G_\sigma(x) \right) \Big / \big(a_0 G_\sigma(x) \big)
& = \frac{1}{a_0} \int_0^\infty \tilde{h}(y) \left( e^{(2xy-y^2)/(2\sigma^2)} + e^{(-2xy-y^2)/(2\sigma^2)} \right) \; dy\\
& = \frac{2}{a_0} \int_0^\infty \tilde{h}(y) \, e^{-y^2/(2\sigma^2)} \, \cosh \left( \frac{xy}{\sigma^2} \right) \; dy.
\end{split}
\end{equation}
Since $\cosh(xy/\sigma^2)$ grows monotonically without bound in $|x|$ for each fixed $y>0$ and $\sigma>0$, it follows that the ratio \eqref{Qratio} grows monotonically without bound in $|x|$ for fixed $\sigma>0$.  We  conclude that for each fixed $\sigma$, the second term of $Q(x,\sigma)$ eventually surpasses the first in magnitude as $x$ grows in absolute value, and that the first term never subsequently equals the second in magnitude as $|x|$ increases further.  Combining this with our initial observation $Q(0,\sigma)<0$, it follows that $Q(x,\sigma)$ has exactly two zeros in $x$ for each $\sigma>0$.
\end{proof}


\section{Uniqueness of Heat Equation Solutions}
\label{heat}

Our results also yield a uniqueness condition for solutions to the heat equation \eqref{heateqn}.  If it is known that $F(x,t)$ solves \eqref{heateqn} for some initial condition $f \in \cP_\gamma' \cap L^1(\R^d)$, then by Corollary \ref{cor:OneGaussian}(a), both $f$ and $F$ are uniquely determined (up to a multiplicative constant) by the zeros of $F_{xx}(x, t_j)$ for any sequence $\{t_j\}$ of positive reals with a positive or infinite limit point.

As stated in Theorem \ref{thm:heat}, a similar result holds for the zeros of $F$ rather than $F_{xx}$ provided it is known that that the second integral
\begin{equation}
\label{secondint}
a(x) = \int_{-\infty}^x \int_{-\infty}^y f(z) \; dz \, dy
\end{equation}
is in $\cP_\gamma' \cap L^1(\R^d)$.  (In particular this requires $\mu_0(f)=\mu_1(f)=0$.)  Letting $A(x,t)$ be the heat equation solution with initial condition $A(x,0)=a(x)$, Theorem \ref{thm:heat} follows from applying Corollary \ref{cor:OneGaussian}(a) to the zeros of $A_{xx}=F$. 

The condition $a \in \cP_\gamma' \cap L^1(\R^d)$ above cannot be dispensed with.  To see this, let $f_1(x)$ and $f_2(x)$ be distinct anti-symmetric functions that are positive for $x >0$ and negative for $x<0$.  The respective solutions of \eqref{heateqn} with initial conditions given by such $f_1$ and $f_2$ have the same zero set, consisting only of the line $x=0$.  In this case, $f_1$ and $f_2$ have positive first moment, so their respective second integrals $a_1$ and $a_2$, defined as in \eqref{secondint}, are not in $\cP_\gamma' \cap L^1(\R^d)$.

Theorem \ref{thm:heat} appears to be a new type of uniqueness theorem
for the heat equation.  In particular, it requires a type of global
agreement between two functions in order to imply their identity.  In
contrast, most heat equation uniqueness theorems \cite[e.g.]{Lin,
Chen} are based on local agreement to infinite order.


\section*{Acknowledgements}

The authors thank Michael Filaseta and Vladimir Temlyakov for useful discussions, in particular for pointing out the theorem of Schur on irreducibility of Hermite polynomials,  required for the proof of Corollary \ref{cor:OneGaussian}(a).  We also thank David Fried for helpful discussions on the algebraic geometry of analytic functions.

\bibliographystyle{plain}
\bibliography{marr}

\end{document}